\documentclass[11pt]{amsart}
\pdfoutput=1

\usepackage[dvipsnames, table, xcdraw]{xcolor}
\usepackage{upgreek}
\usepackage{amssymb,amsmath,mathtools,amsfonts,amsthm,thmtools}
\usepackage[mathscr]{eucal}
\usepackage{fullpage}
\usepackage{color}
\usepackage{caption}
\usepackage{subcaption}
\usepackage{enumitem}

\usepackage[colorlinks]{hyperref}
\definecolor{headercolor}{RGB}{255,255,240}
\definecolor{mylinkcolor}{RGB}{0,0,255}
\definecolor{mycitecolor}{RGB}{169,169,169}
\definecolor{myurlcolor}{RGB}{255,20,147}
\hypersetup{colorlinks=true, urlcolor=myurlcolor, citecolor=mycitecolor, linkcolor=mylinkcolor, linktoc=page, breaklinks=true}

\usepackage{url} 
\usepackage[numbers]{natbib}
\usepackage{setspace}  
\usepackage{indentfirst} 
\usepackage{listings} 

\usepackage{colonequals} 
\usepackage{nicefrac}

\usepackage{tikz-cd,tikz} 
\usepackage{cleveref} 
\usetikzlibrary{positioning}

\usepackage{float}
\usepackage{adjustbox} 
\newcommand{\PreserveBackslash}[1]{\let\temp=\\#1\let\\=\temp}
\newcolumntype{C}[1]{>{\PreserveBackslash\centering}p{#1}}
\usepackage{diagbox}
\usepackage{multirow}
\usepackage{hhline}
\usepackage{booktabs}
\usepackage{placeins}

\numberwithin{equation}{section} 

\newtheorem{theorem}{Theorem}[section]

\newtheorem{lemma}[theorem]{Lemma}
\newtheorem{coro}[theorem]{Corollary}
\newtheorem{conjecture}[theorem]{Conjecture}
\newtheorem{proposition}[theorem]{Proposition}

\theoremstyle{definition} 
\newtheorem{defn}[theorem]{Definition}
\newtheorem{example}[theorem]{Example}

\newtheorem{remark}[theorem]{Remark}

\def \ifempty#1{\def\temp{#1} \ifx\temp\empty }

\newcommand{\claim}[2][]{%
\ifempty{#1}%
  \par\vspace{2mm}\noindent\underline{Claim}:\;\textit{#2}\\%
\else%
  \par\vspace{2mm}\noindent\underline{Claim~#1}:\;\textit{#2}\\%
\fi
}

\numberwithin{table}{section}
\numberwithin{figure}{section}
\newcommand{\oalpha}{\overline{\alpha}} 

\newcommand{\op}{\overline{\mathfrak{p}}}

\newcommand{\RR}{\mathbf{R}} 
 
\newcommand{\ZZ}{\mathbf{Z}} 
\newcommand{\QQ}{\mathbf{Q}} 
\newcommand{\CC}{\mathbf{C}} 
\newcommand{\FF}{\mathbf{F}} 
 
\newcommand{\Fq}{\FF_q}

\newcommand{\Qbar}{\overline{\mathbf{Q}}}

\newcommand{\OO}{\mathscr{O}}
\newcommand{\mfp}{\mathfrak{p}}
\newcommand{\mfq}{\mathfrak{q}}

\newcommand{\mand}{\text{ and }}

\newcommand{\brk}[1]{ \mathopen{}\left\lbrace #1 \right\rbrace\mathclose{}}

\newcommand{\cdef}[1]{{\color{blue}{\textit{#1}}}}

\newcommand{\GitHub}{\textsc{GitHub }}

\newcommand{\indx}{\mathcal{I}}
\newcommand{\rts}{\mathcal{R}}
\newcommand{\cO}{\mathcal{O}} 
\newcommand{\cP}{\mathcal{P}} 
\newcommand{\vecsp}[2]{%
\ensuremath%
#1\langle#2\rangle%
}

\DeclareMathOperator{\rank}{rk}

\DeclareMathOperator{\Gal}{Gal}

\DeclareMathOperator{\Sym}{Sym} 

\renewcommand{\div}{\operatorname{div}}
\DeclareMathOperator{\Div}{Div}

\DeclareMathOperator{\End}{End}

\DeclareMathOperator{\Stab}{Stab}
\DeclareMathOperator{\Span}{span}

\newcommand*{\nfield}[2][]{\href{https://www.lmfdb.org/NumberField/#2}{{\ifx&#1& #2 \else #1 \fi}}}
\usepackage{underscore} 
\newcommand{\avlink}[1]{\href{http://www.lmfdb.org/Variety/Abelian/Fq/#1}{\texttt{#1}}}

\makeatletter
\newcommand\newtag[2]{#1\def\@currentlabel{#1}\label{#2}}
\makeatother

\definecolor{wlabcol}{rgb}{0,0,0.3}

\newcommand{\wgrouplabel}[1]{\textnormal{\texttt{#1}}}
\newcommand{\mkwl}[1]{\newtag{\wgrouplabel{#1}}{#1}}
\newcommand{\wl}[1]{%
  \begingroup\hypersetup{linkcolor=mylinkcolor}\ref{#1}\endgroup
}

\newcommand{\cclass}[1]{\mathcal{G}_{#1}}
\newcommand{\thmemph}[1]{\textbf{#1}}

\setlist[itemize]{leftmargin=30pt, itemsep=2pt}
\setlist[enumerate]{leftmargin=30pt, itemsep=2pt}

\usetikzlibrary{shapes.geometric, arrows.meta, bending, positioning}
\definecolor{transitive}{rgb}{0.0, 0.8, 0.6}
\definecolor{non-transitive}{rgb}{0.9, 0.17, 0.31}

\tikzstyle{start} = [rectangle, rounded corners, minimum width = 2cm, minimum height = 1cm, draw = blue, fill = white, text centered, very thick]

\tikzstyle{stop} = [rectangle, rounded corners, minimum width = 2cm, minimum height = 1cm, draw = blue, fill = white, text width = 2cm, text centered, very thick]

\tikzstyle{stop-t} = [rectangle, rounded corners, minimum width = 2cm, minimum height = 1cm, draw = transitive, fill = white, text width = 2cm, text centered, very thick]

\tikzstyle{stop-nt} = [rectangle, rounded corners, minimum width = 2cm, minimum height = 1cm, draw = non-transitive, fill = white, text width = 2cm, text centered, very thick]

\tikzstyle{stage} = [rectangle, minimum width = 1.5cm, minimum height = 1cm, draw = blue, fill = white, text centered, very thick]

\tikzstyle{stage-NP} = [rectangle, minimum width = 2.1cm, minimum height = 1cm, draw = white, fill = white, text centered, very thick]

\tikzstyle{tag} = [rectangle, draw = white, fill = white, text centered]

\tikzstyle{decision} = [diamond,  minimum width = 1.6cm, minimum height = 1cm, draw = black, fill = white, text width = 1.9cm, text centered, thick]

\tikzstyle{input} = [trapezium, trapezium left angle = 70, trapezium right angle = 110, draw = cyan, fill = white, text width = 2cm, text centered]

\tikzstyle{arrow} = [thick, ->, >=stealth, gray]

\title[Galois groups of low dimensional abelian varieties over finite fields]{Galois groups of low dimensional abelian\\ varieties over finite fields}
\date{June 29, 2026}
\author{Santiago Arango-Pi{\~n}eros}
\address{Department of Mathematics, University of Massachusetts Amherst, Amherst, MA 01003, USA}
\email{santiago.arango.pineros@gmail.com}
\urladdr{\url{https://sarangop1728.github.io/}}

\author{Sam Frengley} 
\address{Inria and Laboratoire d'Informatique de l'École polytechnique, CNRS, Institut Polytechnique de Paris, Palaiseau, France}
\email{samuel.frengley@inria.fr}
\urladdr{\url{https://samfrengley.github.io/}}

\author{Sameera Vemulapalli}
\address{Department of Mathematics,
  Harvard University}
\email{vemulapalli@math.harvard.edu}
\urladdr{\url{https://web.math.princeton.edu/~sameerav/}}

\allowdisplaybreaks

\begin{document}

\begin{abstract} 
	We consider three isogeny invariants of abelian varieties over finite fields: the Galois group, Newton polygon, and the angle rank. Motivated by work of Dupuy, Kedlaya, and Zureick-Brown, we define a new invariant called the \emph{weighted permutation representation} which encompasses all three of these invariants and use it to study the subtle relationships between them. We use this permutation representation to classify the triples of invariants that occur for abelian surfaces and simple abelian threefolds.
\end{abstract}

\maketitle

\vspace{-5mm}
\setcounter{tocdepth}{1}
\tableofcontents
\vspace{-5mm}

\section{Introduction}
\label{sec:intro}

The purpose of this article is to study the surprisingly subtle interactions between three isogeny invariants of abelian varieties over finite fields. We analyze which triples of invariants may occur and what restrictions they impose on the abelian variety in question. The interaction of these invariants is discussed in a letter from Serre to Ribet \cite[p.~6]{Serre89} and has gained renewed interest following the publication of the database of abelian varieties over finite fields in the LMFDB \cite{DupuyKedlayaRoeVincent22}. The availability of this data has demonstrated that the interaction between these invariants is more intricate than initially thought \cite{DupuyKedlayaRoeVincent21}, prompting the development of more refined invariants to better understand these relationships \cite{DupuyKedlayaZureick-Brown22}. This article makes progress towards that goal. Even though this subject is interesting in its own right, it has applications to the Tate conjecture \cite{Zarhin15, Zarhin22}, monodromy groups of abelian varieties over number fields \cite{Zywina22}, Frobenius distributions \cite{AhmadiShparlinsky10, APBS2023}, and prime number races and Chebyshev biases in the context of function fields \cite{keliher2024}. 

By the Honda--Tate theorem \cite{Tate1966, Honda1968, Tate1971}, the isogeny class of an abelian variety $A$ is determined by its Frobenius polynomial, which is the characteristic polynomial of the Frobenius endomorphism acting on the $\ell$-adic Tate module of $A$ (where $\ell$ is a prime number which is not equal to the characteristic of the base field $\Fq$).
All of our invariants are derived from the Frobenius polynomial; the first invariant is the \cdef{Newton polygon} of the Frobenius polynomial, which determines the $p$-adic valuations of the roots, the second invariant is the \cdef{angle rank} (see \Cref{defn:angle-rank}), which measures the nontrivial multiplicative relations between the roots of the Frobenius polynomial, and finally, we have the \cdef{Galois group} of the Frobenius polynomial as our third invariant. We classify triples of invariants that occur for abelian varieties of dimension $\leq 3$; the dimension $3$ case already exhibits some subtleties that should be expected in general.

In \cite{DupuyKedlayaZureick-Brown22}, the authors noticed that the Galois group, Newton polygon, and angle rank are not independent. The Galois group acts on the $p$-adic valuations of the roots (we visualize each root as a ball of radius proportional to its $p$-adic valuation) and this \cdef{weighted permutation representation} (see \Cref{def:perm-rep}) determines the Galois group, Newton polygon, and angle rank. However, this does \emph{not} imply that the Galois group and Newton polygon determine the angle rank. For example, the isogeny classes of abelian threefolds over $\FF_2$ with LMFDB \cite{lmfdb} labels \avlink{3.2.ac_a_d} and \avlink{3.2.a_a_ad} have the same Newton polygon and Galois group, but different angle ranks. This example illustrates the need to consider more than just the isomorphism class of the Galois group.

\subsection{Statement of main results}
The authors of \cite{DupuyKedlayaZureick-Brown22} defined the \emph{Newton hyperplane representation} of a geometrically simple abelian variety to encode this information. The central contribution of this paper is to reinterpret the Newton hyperplane representation in terms of a \emph{weighted permutation representation} and prove additional constraints upon it (see \Cref{sec:divisor-map}). We leverage these results to classify weighted permutation representations for elliptic curves (\Cref{lemma:main-thm-ecs}), abelian surfaces (\Cref{thm:main-thm-surfaces}) and simple abelian threefolds (\Cref{thm:main-thm-3folds}).

The flowcharts in Figures~\ref{fig:flowchart2-simple}--\ref{fig:flowchart3} distill information from the tables in our main theorems. The purpose of these flowcharts is to serve as a ``user's guide'' to the tables in \Cref{thm:main-thm-surfaces} and \Cref{thm:main-thm-3folds}; in particular, if one has in hand an abelian surface or threefold, then the flowchart rules out certain Galois groups. 

\begin{theorem}
    \label{thm:main-surfaces}
    If $A$ is an abelian surface, then the possible isomorphism classes of the Galois group $G_A$ are determined by \Cref{fig:flowchart2-simple} in the simple case, and by \Cref{fig:flowchart2-non-simple} in the non-simple case. Moreover, every possibility occurs.
\end{theorem}

\begin{theorem}
    \label{thm:main-simple-threefolds}
    If $A$ is a simple abelian threefold, then the possible isomorphism classes of the Galois group $G_A$ are determined by \Cref{fig:flowchart3}. Moreover, every possibility occurs.
\end{theorem}

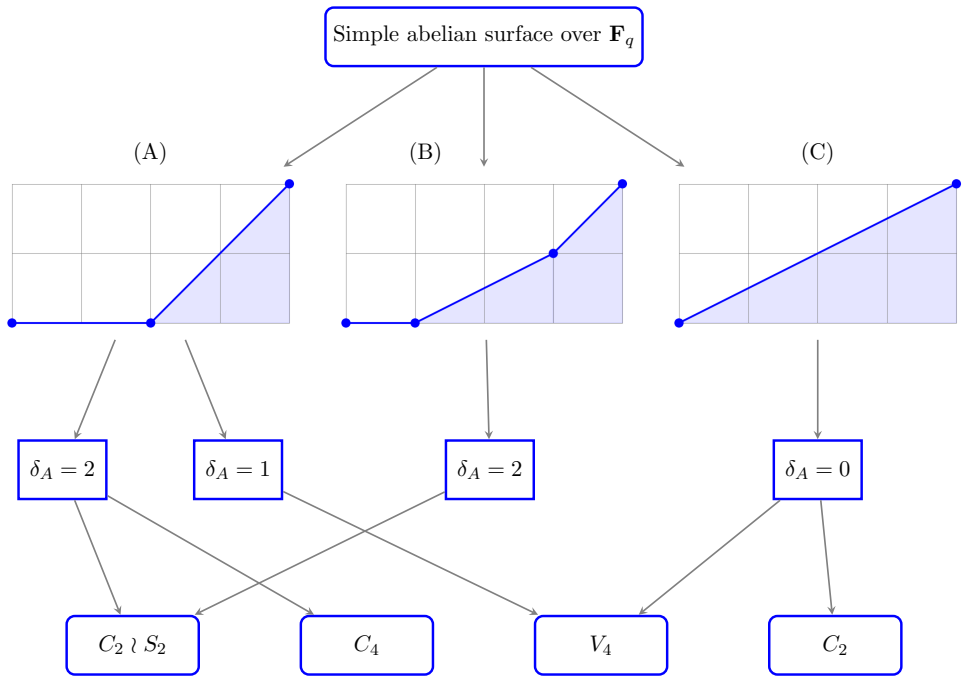
\begin{figure}[p]
    \centering
    \resizebox{0.8\textwidth}{!}{
    \begin{tikzpicture}[node distance = 2.7cm]
    \node (start) [start] {Simple abelian surface over $\mathbf{F}_q$};
    \node (B) [stage-NP, below of = start, yshift=-1cm] {\includegraphics[scale=0.45]{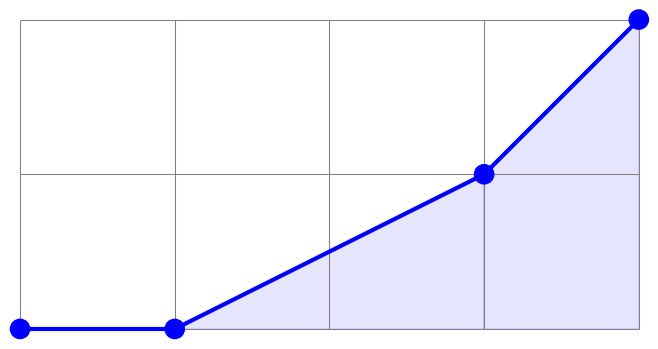}};
    \node (B-tag) [tag, above of=B, yshift=-1cm, xshift=-1cm] {(B)};
    \node (A) [stage-NP, left of = B, xshift = -3cm] {\includegraphics[scale=0.45]{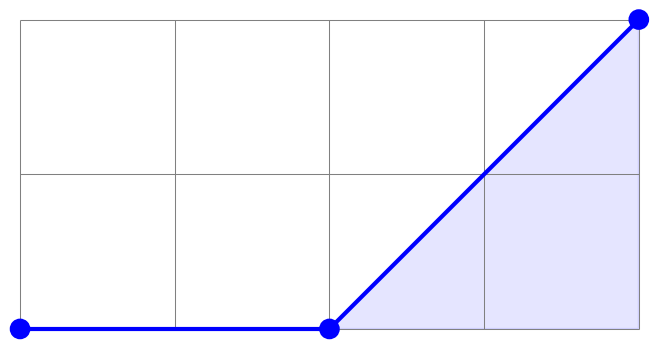}};
    \node (A-tag) [tag, above of=A, yshift=-1cm] {(A)};
    \node (C) [stage-NP, right of = B, xshift = 3cm] {\includegraphics[scale=0.45]{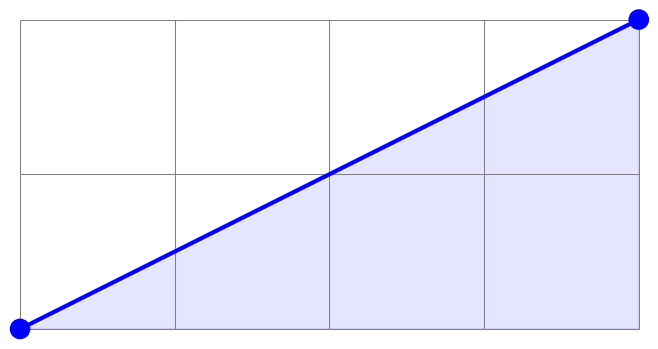}};
    \node (C-tag) [tag, above of=C, yshift=-1cm] {(C)};
    \node (C-delta0) [stage, below of=C,yshift=-1cm] {$\delta_A = 0$};
    \node (A-delta2) [stage, below of=A,xshift=-1.5cm,yshift=-1cm] {$\delta_A = 2$};
    \node (A-delta1) [stage, below of=A,xshift=1.5cm,yshift=-1cm] {$\delta_A = 1$};
    \node (B-delta2) [stage, right of=A-delta1,xshift=1.6cm] {$\delta_A = 2$};
    \node (W4) [stop, below of=B, yshift=-4cm, xshift = -6cm] {$C_2\wr S_2$};
    \node (C4) [stop, below of=B, yshift=-4cm, xshift=-2cm] {$C_4$};
    \node (V4) [stop, below of=B, yshift=-4cm, xshift = 2cm] {$V_4$};
    \node (C2) [stop, below of=B, yshift=-4cm, xshift = 6cm] {$C_2$};
    \draw[arrow] (start) -- (A);
    \draw[arrow] (start) -- (B);
    \draw[arrow] (start) -- (C);
    \draw[arrow] (B) -- (B-delta2);
    \draw[arrow] (C) -- (C-delta0);
    \draw[arrow] (A) -- (A-delta2);
    \draw[arrow] (A) -- (A-delta1);
    \draw[arrow] (B-delta2) -- (W4);
    \draw[arrow] (A-delta2) -- (W4);
    \draw[arrow] (A-delta2) -- (C4);
    \draw[arrow] (A-delta1) -- (V4);
    \draw[arrow] (C-delta0) -- (V4);
    \draw[arrow] (C-delta0) -- (C2);
    \end{tikzpicture}    
    }
    \caption{Possible isomorphism classes of Galois groups of simple abelian surfaces in terms of their Newton polygon, and angle rank $\delta_A$.} \label{fig:flowchart2-simple}
\end{figure}

\begin{figure}[p]
    \centering
    \resizebox{0.8\textwidth}{!}{
    \begin{tikzpicture}[node distance = 2.7cm]
    \node (start) [start] {Non-simple abelian surface $E_1\times E_2$ over $\mathbf{F}_q$};
    \node (B) [stage-NP, below of = start, yshift=-1cm] {\includegraphics[scale=0.45]{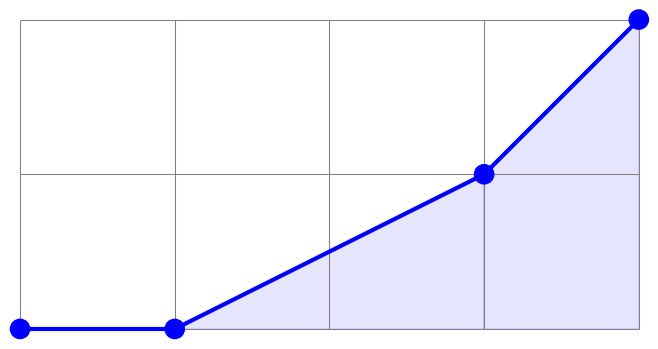}};
    \node (B-tag) [tag, above of=B, yshift=-1cm, xshift=-1cm] {(B)};
    \node (A) [stage-NP, left of = B, xshift = -3cm] {\includegraphics[scale=0.45]{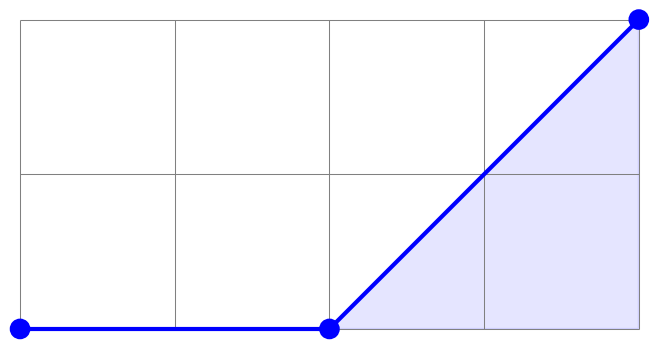}};
    \node (A-tag) [tag, above of=A, yshift=-1cm] {(A)};
    \node (C) [stage-NP, right of = B, xshift = 3cm] {\includegraphics[scale=0.45]{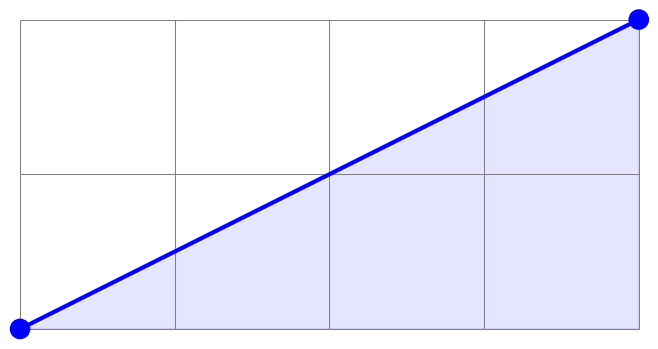}};
    \node (C-tag) [tag, above of=C, yshift=-1cm] {(C)};
    \node (B-delta1) [stage, below of=B,yshift=-1cm] {$\delta_A = 1$};
    \node (C-delta0) [stage, below of=C,yshift=-1cm] {$\delta_A = 0$};
    \node (A-delta2) [stage, below of=A,xshift=-1.5cm,yshift=-1cm] {$\delta_A = 2$};
    \node (A-delta1) [stage, below of=A,xshift=1.5cm,yshift=-1cm] {$\delta_A = 1$};
    \node (V4) [stop, below of=B, yshift=-4cm, xshift = -4cm] {$V_4$};
    \node (C2) [stop, below of=B, yshift=-4cm] {$C_2$};
    \node (C1) [stop, below of=B, yshift=-4cm, xshift = 4cm] {$C_1$};
    \draw[arrow] (start) -- (A);
    \draw[arrow] (start) -- (B);
    \draw[arrow] (start) -- (C);
    \draw[arrow] (B) -- (B-delta1);
    \draw[arrow] (C) -- (C-delta0);
    \draw[arrow] (A) -- (A-delta2);
    \draw[arrow] (A) -- (A-delta1);
    \draw[arrow] (A-delta2) -- (V4);
    \draw[arrow] (A-delta1) -- (C2);
    \draw[arrow] (B-delta1) -- (V4);
    \draw[arrow] (B-delta1) -- (C2);
    \draw[arrow] (C-delta0) -- (V4);
    \draw[arrow] (C-delta0) -- (C2);
    \draw[arrow] (C-delta0) -- (C1);
    \end{tikzpicture}    
    }
    \caption{Possible isomorphism classes of Galois groups of non-simple abelian surfaces in terms of their Newton polygon, and angle rank $\delta_A$.} \label{fig:flowchart2-non-simple}
\end{figure}

\begin{figure}
    \centering
    \resizebox{\textwidth}{!}{
    \begin{tikzpicture}[node distance = 2.7cm]
    \node (start) [start] {Simple abelian threefold over $\mathbf{F}_q$};
    \node (C) [stage-NP, below of = start, yshift=-1cm] {\includegraphics[scale=0.36]{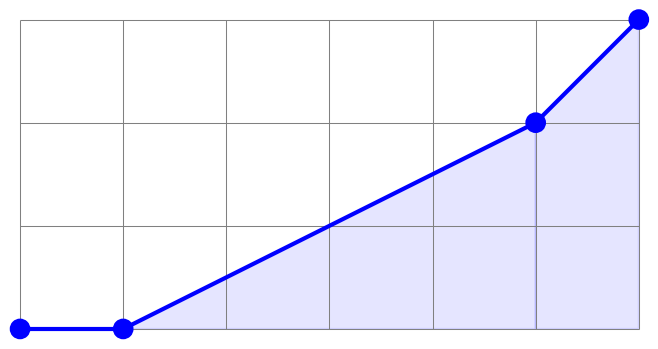}};
    \node (C-tag) [tag, above of=C, yshift=-1.3cm] {(C)};
    \node (B) [stage-NP, left of = C, xshift = -1.5cm] {\includegraphics[scale=0.36]{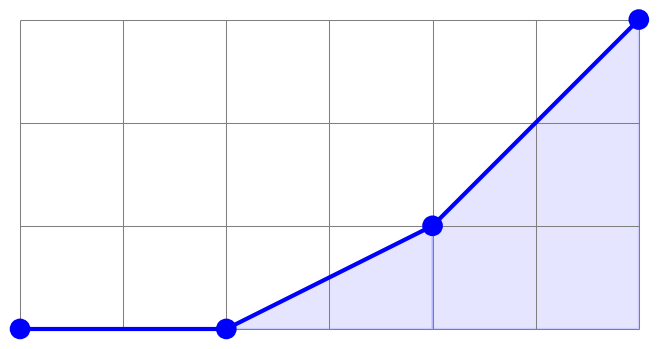}};
    \node (B-tag) [tag, above of=B, yshift=-1.3cm] {(B)};
    \node (A) [stage-NP, left of = B, xshift = -1.5cm] {\includegraphics[scale=0.36]{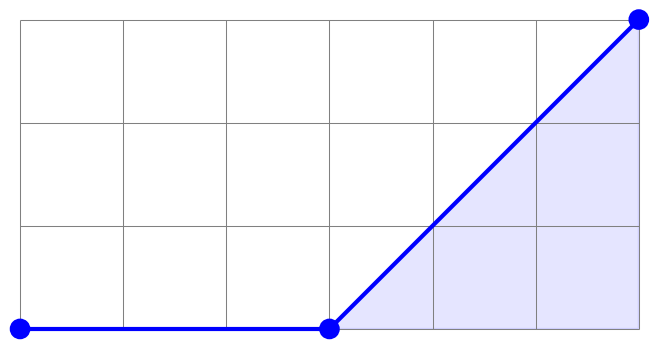}};
    \node (A-tag) [tag, above of=A, yshift=-1.3cm] {(A)};
    \node (D) [stage-NP, right of = C, xshift = 1.5cm] {\includegraphics[scale=0.36]{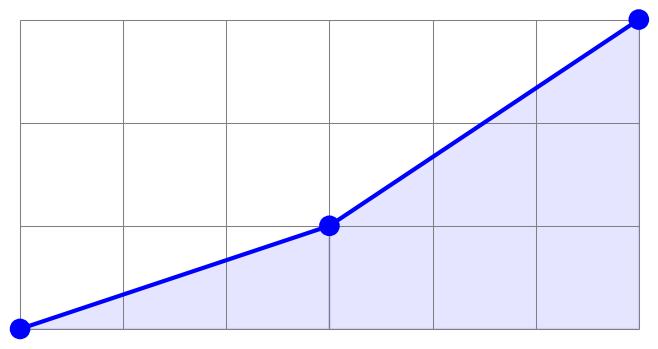}};
    \node (D-tag) [tag, above of=D, yshift=-1.3cm] {(D)};
    \node (E) [stage-NP, right of = D, xshift = 1.5cm] {\includegraphics[scale=0.36]{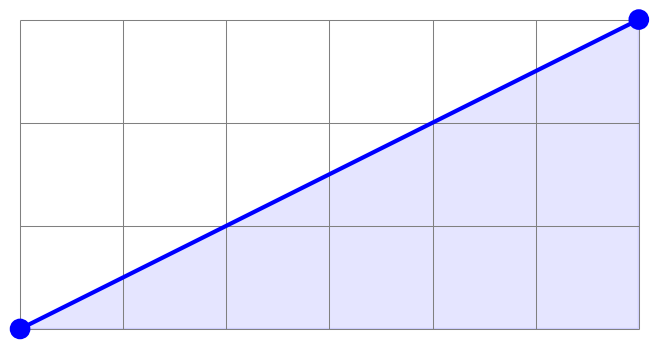}};
    \node (E-tag) [tag, above of=E, yshift=-1.3cm] {(E)};
    \node (A-delta3) [stage, below of = A, xshift = -1cm,yshift=-1cm]{$\delta_A = 3$};
    \node (A-delta1) [stage, below of = A, xshift=1cm,yshift=-1cm]{$\delta_A = 1$};
    \node (B-delta3) [stage, below of = B, xshift = -1cm,yshift=-1cm]{$\delta_A = 3$};
    \node (B-delta2) [stage, below of = B, xshift = 1cm,yshift=-1cm]{$\delta_A = 2$};
    \node (C-delta3) [stage, right of = B-delta2, xshift=0.5cm]{$\delta_A = 3$};
    \node (D-delta3) [stage, below of = D, xshift = -1cm,yshift=-1cm]{$\delta_A = 3$};
    \node (D-delta1) [stage, below of = D, xshift = 1cm,yshift=-1cm]{$\delta_A = 1$};
    \node (E-delta0) [stage, right of = D-delta1]{$\delta_A = 0$};
    \node (W6) [stop, below of = C, yshift = -6cm, xshift = -6cm]{$C_2\wr S_3$};
    \node (C2wrC3) [stop, below of = C, yshift = -6cm, xshift = -2cm]{$C_2\wr C_3$};
    \node (D6) [stop, below of = C, yshift = -6cm, xshift = 2cm]{$D_6$};
    \node (C6) [stop, below of = C, yshift = -6cm, xshift = 6cm]{$C_6$};
    \draw[arrow] (start) -- (A-tag);
    \draw[arrow] (start) -- (B-tag);
    \draw[arrow] (start) -- (C-tag);
    \draw[arrow] (start) -- (D-tag);
    \draw[arrow] (start) -- (E-tag);
    \draw[arrow] (E) -- (E-delta0);

    \draw[arrow] (A) -- (A-delta3);
    \draw[arrow] (A) -- (A-delta1);
    \draw[arrow] (B) -- (B-delta3);
    \draw[arrow] (B) -- (B-delta2);
    \draw[arrow] (C) -- (C-delta3);
    \draw[arrow] (D) -- (D-delta3);
    \draw[arrow] (D) -- (D-delta1);
    \draw[arrow] (A-delta3) .. controls + (down:2cm) .. (W6);
    \draw[arrow] (A-delta3) .. controls + (down:2cm) .. (C2wrC3);
    \draw[arrow] (A-delta3) .. controls + (down:2cm) .. (D6);
    \draw[arrow] (A-delta3) .. controls + (down:2cm) .. (C6);
    \draw[arrow] (A-delta1) .. controls + (down:1cm) .. (D6);
    \draw[arrow] (A-delta1) .. controls + (down:1cm) .. (C6);
    \draw[arrow] (B-delta3) -- (W6);
    \draw[arrow] (B-delta3) -- (C2wrC3);
    \draw[arrow] (B-delta2) -- (C6);
    \draw[arrow] (C-delta3) -- (W6);
    \draw[arrow] (C-delta3) -- (C2wrC3);
    \draw[arrow] (C-delta3) -- (D6);
    \draw[arrow] (D-delta3) -- (W6);
    \draw[arrow] (D-delta3) -- (C2wrC3);
    \draw[arrow] (D-delta1) -- (D6);
    \draw[arrow] (D-delta1) -- (C6);
    \draw[arrow] (E-delta0) -- (C6);
    \end{tikzpicture}    
    }
    \caption{Possible isomorphism classes of Galois groups of simple abelian threefolds, in terms of their Newton polygon and angle rank $\delta_A$.} \label{fig:flowchart3}
\end{figure}

\subsection{Outline}
In \Cref{sec:background} we introduce background and notation used throughout this paper. A reader familiar with abelian varieties may choose to skip directly to \Cref{sec:weighted-perm-rep}, where we introduce a key tool in our paper, the weighted permutation representation. In \Cref{sec:warmup} we warm up by classifying weighted permutation representations of elliptic curves. In \Cref{sec:surfaces} and \Cref{sec:threefolds} we classify weighted permutation representations of abelian surfaces and simple abelian threefolds respectively. We finish by listing further inverse Galois questions in \Cref{sec:questions}.

The code associated to this article is written in \texttt{Magma}~\cite{Magma} and is publicly available from the \GitHub repository \cite{OurElectronic}.

\subsection{Acknowledgements} 
We are grateful to Deewang Bhamidipati, Raymond van Bommel, Soumya Sankar, John Voight, and David Zureick-Brown for helpful conversations about this project. We also thank Everett Howe for several useful correspondences and the anonymous referee for their careful reading of the article. SF was supported by the Woolf Fisher and Cambridge Trusts through a Woolf Fisher scholarship, by C{\'e}line Maistret's Royal Society Dorothy Hodgkin Fellowship, and also by the HYPERFORM consortium, funded by France through Bpifrance. SV was supported by the National Science Foundation under grant number DMS2303211. 

\section{Background and notation}\label{sec:background}

\subsection{Honda--Tate theory}\label{sec:HT-theory} 
Let $A/\FF_q$ be an abelian variety. A celebrated theorem of Honda and Tate classifies isogeny classes of abelian varieties over finite fields. Let $g \colonequals \dim(A) > 0$ and recall that $P_A(T) \in \ZZ[T]$ is the characteristic polynomial of Frobenius. 

The roots of $P_A(T)$ have absolute value $\sqrt{q}$ in all their complex embeddings; an algebraic integer with this property is called  \cdef{$q$-Weil number}. A \cdef{$q$-Weil polynomial} is a monic integral polynomial whose roots are all $q$-Weil numbers. Fix an algebraic closure $\Qbar$ of $\QQ$ inside of $\CC$, and an embedding $\Qbar \hookrightarrow \Qbar_p$. We write $\nu$ for the $p$-adic valuation on $\Qbar_p$, normalized so that $\nu(q) = 1$. 

The statement below is the one presented in \cite[Theorem 4.2.12]{Poonen06}.

\begin{theorem}[Honda--Tate Theorem]
\label{thm:HT}
\noindent
\begin{enumerate}
\item \label{enum:HT1}
  If $A$ is a simple abelian variety, then $P_A(T) = h_A(T)^e$ for some irreducible polynomial $h_A(T) \in \mathbb{Z}[T]$ and some $e \geq 1$.
\item \label{enum:HT2}
  There is a bijection between isogeny classes of simple abelian varieties over $\mathbb{F}_q$ and conjugacy classes of $q$-Weil numbers.
\item \label{enum:HT3}
  Given a $\Gal(\Qbar/\QQ)$-conjugacy class of a $q$-Weil number, let $h_A(T)$ be the minimal polynomial of any element of this conjugacy class. Then there exists a unique integer $e_A \colonequals e \geq 1$ such that $h_A(T)^e = P_A(T)$ for some simple abelian variety $A$ over $\Fq$. Moreover, $e$ is the smallest positive integer such that:
  \begin{enumerate}
  \item $h_A(0)^e > 0$, and
  \item For each monic $\QQ_p$-irreducible factor $g(T) \in \QQ_p[T]$ of $h_A(T)$, the valuation $\nu(g(0)^e)$ is in $\ZZ$.
  \end{enumerate}
\end{enumerate}
\end{theorem}

The polynomial $h_A(T)$ is the \cdef{minimal polynomial of the $q$-Frobenius endomorphism of $A$}. When $P_A(T)$ is totally complex (i.e., has no real roots), the degree of $h_A(T)$ is equal to $2d$ for some positive integer $d$. 

In general, the isogeny factorization of $A$ yields a factorization of the Frobenius polynomial. In particular, by the Honda--Tate theorem, two abelian varieties $A$ and $B$ are isogenous if and only if $P_A(T)$ is equal to $P_B(T)$. 

This observation is sufficient to understand the classification of our three isogeny invariants in the case of elliptic curves (see \Cref{sec:warmup}). We now proceed to give more background that will be useful for higher dimensional abelian varieties.

We will use the following as a running example throughout the article to consolidate our definitions and notations. This example appears in \cite[Example 6.1]{Shioda81} and \cite[Example 1.7]{Zywina22} (see also \cite[Example 4.2.9]{Gallese2024}).
\begin{example}[Shioda's example]
    \label{example:Shioda}
    Let $q = p = 19$, and let $A$ be the Jacobian of the hyperelliptic curve $C$ with affine equation $y^2 = x^9 - 1$, defined over the field $\FF_{19}$. The curve $C$ has genus $g = 4$ and therefore $A$ is an abelian fourfold. By calculating $\#C(\FF_{19^r})$ for $r=1,2,3,4$, we are able to estimate the zeta function of $C$ to enough precision to recover the Frobenius polynomial $P_A(T)$. It is given by:
    \begin{equation*}
        P_A(T) = T^8 + 8T^7 + 28T^6 + 8T^5 - 170T^4 + 152T^3 + 10108T^2 + 54872T + 130321.
    \end{equation*}
    This polynomial factors as $P_A(T) = P_E(T)P_B(T)$, where $E$ is the elliptic curve $y^2 = x^3-1$ in the isogeny class \avlink{1.19.i} with Frobenius polynomial $P_E(T) = T^2 + 8T + 19$, and $B$ is an abelian threefold in the isogeny class \avlink{3.19.a_j_acm} with Frobenius polynomial $P_B(T) = T^6 + 9T^4 - 64T^3 + 171T^2 + 6859$. By the Honda--Tate theorem $A$ is isogenous to the product $E \times B$.
\end{example}

\subsection{Newton polygons of Frobenius polynomials}
Let $A$ be a $g$-dimensional abelian variety over $\Fq$. At this moment, we do not assume that $A$ is simple. 
\begin{defn}[$q$-Newton polygon]
  \label{def:NP}
The \cdef{$q$-Newton polygon} of $A$ is the $\nu$-adic Newton polygon of $P_A(T)$. More precisely, if $P_A(T) = \sum_{j = 0}^{2g} a_{2g-j}T^j$, then the $q$-Newton polygon of $A$ is the lower convex hull of the finite set
\begin{equation*}
    \brk{(j,\nu(a_j)) : 0 \leq j \leq 2g,\mand a_j \neq 0} \subset \RR^2.
\end{equation*}
\end{defn}

\begin{example}
    Continuing with \Cref{example:Shioda}, we note that both $E$ and $B$ are ordinary varieties. Since $A \sim E \times B$, the Newton polygon of $A$ is also ordinary, and it is obtained by concatenating those of $E$ and $B$. Alternatively, one can notice that $-170$ (the middle coefficient of $P_A(T)$) is not divisible by $p=19$.
\end{example}

\subsection{Angle rank of Frobenius polynomials}
We recall the following definition from \cite{DupuyKedlayaRoeVincent22,DupuyKedlayaZureick-Brown22,APBS2023}.

\begin{defn}
  \label{defn:angle-rank}
  Consider the multiplicative subgroup $U_A \subset \Qbar^\times$ generated by the normalized Frobenius eigenvalues $u \coloneqq \alpha/\sqrt{q}$ where $\alpha$ ranges over the roots of $h_A(T)$.
  The \cdef{angle rank of $A$} is denoted $\delta_A$ and is defined to be the dimension of $U_A \otimes \QQ$. 
\end{defn}

\subsection{Galois groups of Frobenius polynomials} Given an abelian variety $A$, denote by $R_A$ the \cdef{set of roots}, without multiplicity, of the Frobenius polynomial $P_A(T)$.

\begin{defn}[Galois group]
\label{def:gal-grp}
    The \cdef{Galois group} of $A$ is the Galois group of the minimal polynomial of Frobenius $h_A(T)$. Equivalently, $G_A$ is the largest quotient of $\Gal(\Qbar/\QQ)$ over which the permutation action on $R_A$ factors. 
\end{defn}

\begin{defn}
    \label{def:K}
    In the case that $A$ is simple, we will denote by $K_A$ the center of the endomorphism algebra $\End(A)\otimes\QQ$. We have that $K_A \cong \QQ[T]/(h_A(T))$ is a number field.
\end{defn}

For ``most'' abelian varieties the polynomial $P_A(T)$ is totally complex (i.e., has no real roots), and $K_A$ is a \cdef{complex multiplication} (CM) number field. See \cite{Dodson1984} for some background on Galois groups of CM number fields, and \cite[Section 2.2]{DupuyKedlayaZureick-Brown22} for a discussion on Galois groups of $q$-Weil polynomials. 

\begin{example}
    \label{example:part-2}
    Continuing with \Cref{example:Shioda}, we have that the splitting field of $P_A(T)$ is the degree $6$ field $\QQ(\zeta_9)$, where $\zeta_9$ is a primitive $9^\text{th}$-root of unity. This already implies that the $8$ roots of $P_A(T)$ are algebraically dependent. Recalling that $A \sim E \times B$, observe that $K_E = \QQ(\sqrt{-3})$, which is contained in $K_B = \QQ(\zeta_9)$. Note that $G_A$ is a permutation group acting on the $8$ element set $R_A$. But as abstract groups, $G_A \cong \Gal(\QQ(\zeta_9)/\QQ) \cong C_6$. 

    Denote by $R_E = \brk{\eta, 19\eta^{-1}}$ and $R_B = \brk{\alpha,\beta,\gamma,19\alpha^{-1},19\beta^{-1}, 19\gamma^{-1}}$ the sets of roots of $P_E(T)$ and $P_B(T)$ respectively.    For an appropriate such labelling, we have 
    \begin{equation}
        \label{eq:mult-rel-example}
        \eta = \dfrac{\alpha\cdot\beta\cdot\gamma}{19}.
    \end{equation}
    We will see in \Cref{example:angle-rank} that the $\QQ$-vector space $U_A\otimes \QQ$ has dimension 3, i.e., $\delta_A = 3$.
\end{example}

Instead of viewing $G_A$ as a permutation subgroup of $S_n$ for $n = \# R_A$, it will be convenient to use a ``CM specific'' permutation representation.

\subsection{The group of signed permutations}\label{sec:W2d}
We now describe the abstract group which Galois groups of totally complex $q$-Weil numbers are naturally contained in. First note that when $h_A(T)$ is totally complex of degree $2d$, the roots of $h_A(T)$ come in complex conjugate pairs. Moreover, the action of the Galois group $G_A$ respects this partition. 

Let $X_{2d}$ be the set consisting of the symbols $1, \bar{1}, \dots, d, \bar{d}$. Define $W_{2d}$ to be the subgroup of $\Sym(X_{2d})$ which preserves the partition
\begin{equation*}
    X_{2d} = \brk{1,\bar{1}}\sqcup \cdots \sqcup \brk{d,\bar{d}}.
\end{equation*}
Upon a choice of labelling of roots, the Galois group of $h_A(T)$ may naturally be embedded in $W_{2d}$. We refer to the element $\iota \colonequals (1\bar{1})\ldots (d\bar{d})$ as \cdef{complex conjugation}.

\subsubsection{Subgroup labelling} \label{sec:subgroup-labelling}
We now briefly describe our naming conventions for subgroups of $W_{2d}$. A group $H$ is denoted $\texttt{G.d.t.letter.k}$ if it is isomorphic to $G$, contained in $W_{2d}$, and acts transitively on $X_{2d}$. The $W_{2d}$ conjugacy class is indexed by $\texttt{letter}$, and the groups in that conjugacy class are indexed by the tiebreaker $\texttt{k}$ (a positive integer). The use of $\texttt{nt}$ instead of $\texttt{t}$ indicates that $H$ acts intransitively on $X_{2d}$. For example, the group $\wl{C2.4.nt.c.1}$ refers to a group which is isomorphic to $C_2$, contained in $W_4$, and is intransitive. The label $\texttt{c}$ refers to the conjugacy class in $W_4$ and the index $\texttt{1}$ means that it is the first listed in its conjugacy class.

\begin{remark}
    In the code associated to this article \cite{OurElectronic} we provide functions which compute and label the transitive subgroups of $W_{2d}$ which contain complex conjugation in the file \href{https://github.com/sarangop1728/Galois-Frob-Polys/blob/main/src/W2d-subgroups.m}{\texttt{src/W2d-subgroups.m}}. Our labelling convention is well defined and essentially follows lexicographic ordering of the subgroups of the symmetric group $S_{2d} \supset W_{2d}$ (as described in \cite{HulpkeLinton:TotalOrdering}). See the file \href{https://github.com/sarangop1728/Galois-Frob-Polys/blob/main/src/subgroup-labelling.m}{\texttt{src/subgroup-labelling.m}}.
\end{remark}

\section{The weighted permutation representation}
\label{sec:weighted-perm-rep}
In this section we introduce a key tool in this paper -- the notion of a weighted permutation representation. This construction is heavily inspired by the definition of \emph{Newton hyperplane arrangement} of Dupuy, Kedlaya, and Zureick-Brown \cite{DupuyKedlayaZureick-Brown22}.

\subsection{The weighted permutation representation}
\label{sec:perm-rep} 

We now describe the weighted permutation representation associated to an abelian variety, which is the central isogeny invariant in this article. It determines the Galois group, angle rank, and Newton polygon. The main purpose of our paper is to determine which weighted permutation representations occur from low dimensional abelian varieties. The weighted permutation representation is a reinterpretation of the \emph{Newton hyperplane representation} discussed in \cite{DupuyKedlayaZureick-Brown22}.

\begin{defn}[Weighted permutation representations]
  \label{def:weighted-perm-rep}
  Given a finite group $G$, a \cdef{weighted permutation representation of $G$} is a pair $(w,\rho)$, where $w \colon X_{2d} \rightarrow \QQ_{\geq 0}$ is a map of sets and $\rho \colon G \xhookrightarrow{} W_{2d}$ is an inclusion of groups.

  We say that a pair of weighted permutation representations $(w, \rho)$ and $(w, \rho')$ of $G$ are \cdef{$w$-conjugate} if they are conjugate by an element of $\Stab(w)$, i.e., if there exists an element $\sigma \in W_{2d}$ such that $w = w \circ \sigma$ and $\rho'(g) = \sigma^{-1} \rho(g) \sigma$ for all $g \in G$. 
\end{defn}

\begin{defn}[Roots]
    \label{def:roots}
    Recall that for an abelian variety $A$ we write $R_A$ for the set of roots of $h_A(T)$ without multiplicity. We define $\rts_A$ to be the \cdef{multiset} of roots of $h_A(T)$, that is:
    \begin{enumerate}
    \item
      When $h_A(T)$ is totally complex, $\rts_A = R_A$.
    \item
      When $h_A(T)$ has a real root $\alpha$, it is counted with multiplicity two, and its duplicate is denoted by $\oalpha$.
    \end{enumerate}
    We define $d_A = \#\rts_A / 2$.
\end{defn}

\begin{defn}[Indexing and weighting]
    \label{def:indexing}
    An \cdef{indexing of roots} of $h_A(T)$ is a bijection $\indx\colon X_{2d} \to \rts_A$ which satisfies the following conditions:
    \begin{enumerate}
    \item
      $\indx$ respects complex conjugation, i.e., $\indx(\overline{k}) = \overline{\indx(k)}$ for each $1 \leq k \leq d$, and
    \item
     the indices climb the Newton polygon, i.e.,
      \[
        \nu(\indx(i)) \leq \nu(\indx(j)) \leq \nu(\indx(\bar{j})) \leq \nu(\indx(\bar{j}))
      \]
      for each pair of indices $1 \leq i \leq j \leq d$.
    \end{enumerate}
\end{defn}

For simplicity, given an indexing $\indx$ we write $\alpha_i \coloneqq \indx(i)$ and $\oalpha_i \coloneqq \indx(\bar{i})$ in what follows. Any indexing of the roots naturally gives a \cdef{weighting} $w_A \colon X_{2d} \rightarrow \QQ_{\geq 0}$ given by $k \mapsto \nu(\alpha_k)$ and $\bar{k} \mapsto \nu(\oalpha_k)$ for $k \in \brk{1,\dots,d}$. We omit the choice of indexing from the notation, since every choice of indexing yields the same weighting $w_A$.  Note that the weighting $w_A$ associated to an abelian variety $A$ is uniquely determined by the $q$-Newton polygon of $A$.

\begin{defn}[Weighted permutation representation, totally complex case]
    \label{def:perm-rep}
    Suppose $h_A(T)$ is totally complex. Given an indexing $\mathcal{I}$, we obtain a representation 
    \begin{equation*}
        \rho_{\mathcal{I}} \colon G_A \xhookrightarrow{} \Sym(X_{2d}) \cong S_{2d}    
    \end{equation*}
    whose image lies in $W_{2d}$. The pair $(w_{A}, \rho_{\mathcal{I}})$ is the \cdef{weighted permutation representation} associated to $A$ with respect to the indexing $\mathcal{I}$.
\end{defn}

The definition above \emph{canonically} defines the weighted permutation representation associated to $A$ up to $w_A$-conjugacy when $P_A(T)$ is totally complex. The following definition gives a notion of weighted permutation representation when $P_A(T)$ is totally real; although the definition does not appear canonical, we will show later that it is in fact unique up to isomorphism.

The simple isogeny classes of abelian varieties with real eigenvalues are highly constrained. 
\begin{lemma}[{\cite[p.~528]{Waterhouse1969}}]
    \label{lem:real-eigenvalues}
    Let $A$ be a simple abelian variety whose Frobenius eigenvalues are real. Then:
    \begin{enumerate}[label=(\arabic*)]
    \item \label{enum:real-ECs}
      If $q$ is a square, $A$ is a supersingular elliptic curve, and $h_A(T) = (T \pm \sqrt{q})$ for some choice of sign, or
    \item \label{enum:complex-ECs}
      If $q$ is not a square, $A$ is a supersingular abelian surface, and $h_A(T) = (T^2-q)$.
    \end{enumerate}
    Moreover, in case~\ref{enum:real-ECs} we have $\rts_A = \brk{\sqrt{q}, \overline{\sqrt{q}}}$ or $\brk{-\sqrt{q}, -\overline{\sqrt{q}}}$, and in case~\ref{enum:complex-ECs} we have $\rts_A = \brk{\sqrt{q}, \overline{\sqrt{q}}, -\sqrt{q}, -\overline{\sqrt{q}}}$.
\end{lemma}

\begin{defn}[Weighted permutation representation, totally real case]
    \label{def:perm-rep-real}
    Suppose that $A$ has only real Frobenius eigenvalues.
    \begin{enumerate}
    \item[(1a)] If $q$ is a square and $A$ has one eigenvalue, then $G_A$ is trivial. Define $\rho_{\mathcal{I}} \colon G_A \to W_2$ to be the trivial map for all indexings.
    \item[(1b)] If $q$ is a square and $A$ has two eigenvalues, then $G_A$ is trivial. Define $\rho_{\mathcal{I}} \colon G_A \to W_4$ to be the trivial map for all indexings.
    \item[(2)] If $q$ is not a square, then $G_A \cong C_2$. Define $\rho_{\mathcal{I}} \colon G_A \to W_4$ to be the homomorphism sending the nontrivial element on $G_A$ to $(12)(\bar{1}\bar{2})$ for all indexings.
    \end{enumerate}
    The pair $(w_A, \rho_\mathcal{I})$ is the \cdef{weighted permutation representation} associated to $A$ with respect to the indexing $\mathcal{I}$. 
\end{defn}

We now define the weighted permutation representation associated to an abelian variety $A$ for which $h_A(T)$ need not be totally real or totally complex; it is essentially the direct sum of its totally real and totally complex parts. In this case $A$ is isogenous to a product $B \times C$ where $h_B(T)$ is totally real and $h_C(T)$ is totally complex. Let $(w_B,\rho_B)$ and $(w_C,\rho_C)$ be the weighted permutation representations corresponding to the factors $B$ and $C$ respectively. Let $(w_B \oplus w_C, \rho_B \oplus \rho_C)$ be the direct sum of these weighted permutation representations; conjugate by $W_{2d}$ so that the weight function is nondecreasing, i.e., so that
\[
    w(i) \leq w(j) \leq w(\bar{j}) \leq w(\bar{i})
\]
for $i \leq j$. We define the weighted permutation representation associated to $A$ to be the resulting weighted permutation representation. \Cref{lemma:well-def-conj-class} follows by construction.

\begin{lemma}
\label{lemma:well-def-conj-class}
    The weighted permutation representation associated to an abelian variety $A$ is well defined up to $w_A$-conjugacy (irrespective of indexing). In particular, its image in $W_{2d}$ is a subgroup $\cclass{A}$ which is well-defined up to $w_A$-conjugacy.
\end{lemma} 

\subsection{The angle rank}\label{sec:newton-rep}

It turns out that the angle rank can be computed from the weighted permutation representation of an abelian variety in the following way. 

\begin{defn}
 Given a weighted permutation representation $(w,\rho \colon G \xhookrightarrow{} W_{2d})$, define the \cdef{angle rank of $(w,\rho)$} to be $\rank(M) - 1$ where $M$ is the $(d \times |G|)$-matrix whose $i^\text{th}$-column has entries $w(\sigma(i))$ where $\sigma$ ranges over elements of $G$.
\end{defn}

The angle rank of $(w,\rho)$ is equal to the rank of the $(d \times |G|)$-matrix whose $i^\text{th}$-column has entries $w(\sigma(i)) - w(\sigma(\bar{i}))$, which is referred to as the Newton hyperplane matrix \cite[Remark 3.4]{DupuyKedlayaZureick-Brown22}. We show in \Cref{prop:newton-matrix-rank} that the angle rank of an abelian variety is equal to the angle rank of its weighted permutation representation.

\begin{example}\label{example:angle-rank}
Continuing with \Cref{example:part-2}, fix a prime $\mfp$ of the splitting field $K = \QQ(\zeta_9)$ above $p = 19$, and let $\nu$ be the extension of the $19$-adic valuation of $\QQ$ extended to $K$. We fix the following indexing of the roots $R_A$ (see \Cref{def:indexing}):
\begin{equation}
    \alpha_1 := \alpha, \quad \alpha_2 := \beta, \quad \alpha_3 := \eta, \quad \alpha_4 := \gamma,
\end{equation}
if we ensure that $\nu(\alpha) = \nu(\beta) = \nu(\gamma) = \nu(\eta) = 0$. With this indexing, the weighted permutation representation $\rho\colon \Gal(\QQ(\zeta_9)/\QQ) \to W_8$ has image $H = \langle h \rangle$, where $h = (1\bar{2}\bar{4}\bar{1}24)(3\bar{3})$. With respect to this indexing, the multiplicative relation in \Cref{eq:mult-rel-example} becomes
\begin{equation}
    \alpha_3 = \dfrac{\alpha_1\alpha_2\oalpha_4}{19} =\dfrac{\alpha_1\alpha_2}{\alpha_4}.
\end{equation}
One can find this multiplicative relation by considering the Newton hyperplane matrix of this weighted permutation representation, and computing its kernel. By direct calculation we see that the rank of the Newton hyperplane matrix (which is equal to $\delta_A$) is three.
\end{example}

\subsection{The divisor map} \label{sec:divisor-map}
Many of the proofs in this paper make use of the same key idea, which we elucidate here. Given a simple totally complex $g$-dimensional abelian variety $A$ with Galois group $G_A$, we may construct the following. Let $L$ be the Galois closure of the field $K$, i.e., the field generated by the roots of $P_A(T)$. Let $\cP_L$ be the set of primes in $L$ above $p$. Let $\vecsp{\QQ}{R_A}$ denote the $\QQ$-vector space of dimension $2d$ whose basis consists of the formal symbols $[\alpha]$ where $\alpha$ ranges over the set of Frobenius eigenvalues.

Let $\Div_{\QQ}(\OO_L)$ be the free $\QQ$-module supported on the prime ideals of $\OO_L$ and let $\div_A \colon \OO_L \to \Div_{\QQ}(\OO_L)$ be the map sending each element $x \in \OO_L$ to $ \sum_{\mfp} a_{\mfp} \mfp$ where $(x) = \prod_\mfp \mfp^{a_{\mfp}}$ is the prime factorization of the ideal generated by $x$. By abuse of notation we write
\[
    \div_A \colon \vecsp{\QQ}{R_A} \rightarrow \vecsp{\QQ}{\cP_L}
\]
for the $\QQ$-linear map given by linearly extending $\div_A$ on the roots $\alpha \in R_A$. 

Note that $\vecsp{\QQ}{R_A}$ naturally inherits the structure of a $G_A$-module from the action of $G_A$ on $R_A$, and similarly $\vecsp{\QQ}{\cP_L}$ has an action of $G_A$ inherited from the action of $G_A$ on $\cP_L$. The map $\div_A$ is $G_A$-equivariant. Note that $\div_A$ determines the weighted permutation representation, and thus also determines the angle rank, Newton polygon, and Galois group.

It is natural to ask which $G_A$-module homomorphisms are permissible as the divisor map of a simple abelian variety. We list some necessary conditions which follow immediately from the construction.

\begin{proposition}
\label{prop:div-properties}
    Let $A$ be a simple abelian variety with no real Frobenius eigenvalues. Then:     
    \begin{enumerate}[label=(\arabic*)]
    \item \label{enum:div-property-1} $\div_A([\alpha] + [\overline{\alpha}]) = \div_A(q)$ for all $\alpha \in R_A$;
    \item \label{enum:div-property-2} the $G_A$-action on $R_A$ is transitive;
    \item \label{enum:div-property-3} the $G_A$-action on $\cP_L$ is transitive; and
    \item \label{enum:div-property-4} if the Newton polygon has a segment of horizontal length $m$ which contains exactly two lattice points, then $\# \cP_L \mid \frac{1}{m}\# G_A$.
    \end{enumerate}
\end{proposition}
\begin{proof}
Claims \ref{enum:div-property-1}--\ref{enum:div-property-3} follows from the fact that $\alpha \overline{\alpha} = q$ and the assumption that $A$ is simple.

To prove claim \ref{enum:div-property-4}, observe that if the Newton polygon has a slope of length $m$, then by the Honda--Tate theorem, the field $K$ contains a prime $\mfq$ above $p$ of degree divisible by $m$ (i.e., the product of the ramification index and the inertia degree is divisible by $m$). Because $L$ is the Galois closure of $K$, the degree of any prime in $\cO_L$ must be divisible by $m$, so $\# \cP_L \mid \frac{1}{m}\# G_A$.
\end{proof}

We now quickly use the divisor map to show the following lemma.

\begin{lemma}
\label{prop:newton-matrix-rank}
The angle rank $\delta_A$ of an abelian variety is equal to the angle rank of its weighted permutation representation. Moreover, $\delta_A = \rank(\div_A) - 1$.
\end{lemma}
\begin{proof}
Consider the subspace $V_p = \Span(\div_A(p)) \subset \vecsp{\QQ}{\cP_L}$. We first claim that the angle rank of an abelian variety is precisely the rank of the linear map $\vecsp{\QQ}{R_A} \rightarrow \vecsp{\QQ}{\cP_L}/ V_p$ induced by $\div_A$. Clearly a multiplicative relation between the Frobenius eigenvalues gives rise to an element in the kernel of this map. Conversely, given an element $\sum_{i = 1}^{2g} k_i\alpha_i$ in the kernel, we have
\[
    \bigg (\prod_i\alpha_i^{k_i}\bigg ) =  (q)^{\sum k_i /2},
\]
so there exists $u \in \OO_L^{\times}$ such that
\[
    \prod_iq^{-k_i/2}\alpha_i^{k_i} = u.
\]
Because $u$ has length $1$ in all complex embeddings, $u$ is a root of unity. Now observe that the rank of $\vecsp{\QQ}{R_A} \rightarrow \vecsp{\QQ}{\cP_L} / V_p$ is precisely one less than the rank of $\div_A \colon \vecsp{\QQ}{R_A} \rightarrow \vecsp{\QQ}{\cP_L}$. 

Take the matrix representing the latter and make the following operations. First, for each complex conjugate pair of roots in $R_A$, remove the row corresponding to precisely one of the roots. Next, for each element of $\cP_L$, replace the column corresponding to it with $\# G_A /\#\cP_L$ copies of itself.

Observe that the operations above do not change the rank. Now, after possibly permuting the rows and columns, the matrix is precisely the matrix given in the definition of the angle rank of a weighted permutation representation.
\end{proof}

Observe that for any positive integer $n$, every map $\vecsp{\QQ}{X_{2d}} \rightarrow \QQ^{n}$ satisfying the properties above gives rise to a weighted permutation representation. To prove that a certain weighted representation cannot occur, it suffices to show that there does not exist a lift $\QQ^{2d} \rightarrow \QQ^{n}$ satisfying the conditions above. 

\vspace{-4mm}
\section{Warm-up: elliptic curves}\label{sec:elliptic-curves}
\label{sec:warmup}
We begin with the case when $A$ is an elliptic curve. In accordance with the labelling convention described in \Cref{sec:W2d} we write $\texttt{W2.2.t.a.1}$ and $\texttt{C1.2.t.a.1}$ for the subgroups of order $2$ and $1$ of $W_2$. The following \namecref{lemma:main-thm-ecs} is immediate.

\begin{proposition}
  \label{lemma:main-thm-ecs}
  Let $A$ be an elliptic curve. Then the image, $\cclass{A}$, of the weighted permutation representation associated to $A$ is equal to $W_2$ except when $q$ is a square and $P_A(T) = (T \pm \sqrt{q})^2$, in which case $A$ is supersingular and $\cclass{A} = \{ \operatorname{id} \}$. More precisely, the possible combinations of Galois group, Newton polygon, and angle rank that occur are displayed in Tables~\ref{tab:EC-A}~and~\ref{tab:EC-B}.
\end{proposition}

\begin{table}[H]
  \setlength{\arrayrulewidth}{0.3mm} 
  \setlength{\tabcolsep}{5pt}
  \renewcommand{\arraystretch}{1.2}
  \centering
  \begin{tabular}{|c|c|c|c|}
    \hline
    \rowcolor{headercolor}
    $w_A$-conjugacy class  & Angle rank & Occurs  & Example(s)      \\
    \hline
    $\texttt{W2.2.t.a.1}$  & $1$        & Yes                      & \avlink{1.2.ab} \\ \hline
    $\texttt{C1.2.nt.a.1}$ & $0$        & No                         &                 \\ \hline
  \end{tabular}
    \caption{The $w_A$-conjugacy classes of subgroups $G \subset W_2$ which occur as the image of the weighted permutation representation associated to an ordinary elliptic curve.}
  \label{tab:EC-A}
\end{table}

\begin{table}[H]
  \setlength{\arrayrulewidth}{0.3mm} 
  \setlength{\tabcolsep}{5pt}
  \renewcommand{\arraystretch}{1.2}
  \centering
  \begin{tabular}{|c|c|c|c|}
    \hline
    \rowcolor{headercolor}
    $w_A$-conjugacy class  & Angle rank & Occurs & Example(s)      \\
    \hline
    $\texttt{W2.2.t.a.1}$  & $0$        & Yes                      & \avlink{1.2.ac} \\ \hline
    $\texttt{C1.2.nt.a.1}$ & $0$        & Yes                      & \avlink{1.4.ae} \\ \hline
  \end{tabular}
    \caption{The $w_A$-conjugacy classes of subgroups $G \subset W_2$ which occur as the image of the weighted permutation representation associated to a supersingular elliptic curve.}
  \label{tab:EC-B}
\end{table}

\FloatBarrier

\section{Abelian surfaces}\label{sec:surfaces}
In this section we prove the following theorem.

\begin{theorem}
  \label{thm:main-thm-surfaces}  
  Let $A$ be an abelian surface.
  \begin{itemize}
      \item \thmemph{Permutation representations in $W_4$:} If $A$ has permutation representation contained in $W_4$, then the possible images $\cclass{A}$ of the weighted permutation representation associated to $A$ are given in Tables~\ref{tab:dim2-A}~and~\ref{tab:dim2-A-ns} when $A$ is ordinary, in Tables~\ref{tab:dim2-B}~and~\ref{tab:dim2-B-ns} when $A$ is almost ordinary, and in Tables~\ref{tab:dim2-C}~and~\ref{tab:dim2-C-ns} when $A$ is supersingular. The weighted permutation representation determines whether $A$ is geometrically simple; this information can be seen in the tables as well.
      \item \thmemph{Permutation representations in $W_2$ when $A$ is simple:}
      If $A$ has a permutation representation contained in $W_2$ and is simple, then $A$ is supersingular and hence the angle rank is $0$. Both trivial Galois group and Galois group $C_2$ occur. $A$ is not geometrically simple.
      \item  \thmemph{Permutation representations in $W_2$ when $A$ is not simple:} Now suppose $A$ has a permutation representation contained in $W_2$ and is not simple. Then $A$ is isogenous (over $\Fq$) to $E^2$ for some elliptic curve $E$ then the weighted permutation representation associated to $A$ takes values in $W_2$, and its image is equal to that associated to $E$ (and classified in \Cref{lemma:main-thm-ecs}).
  \end{itemize}
\end{theorem}

The rest of the section is dedicated to proving \Cref{thm:main-thm-surfaces}. The third point follows from the definition and the section on elliptic curves. The second point follows from the classification given in \cite{Xing1994}; see also \cite[Theorem~2.9~(SS2)]{MaisnerNart02}. In the remainder of this section we assume $A$ has permutation representation contained in $W_4$.

\subsection{Permutation representations of subgroups of \texorpdfstring{$W_4$}{W4}}
The group $W_4$ is isomorphic to $D_4$, the dihedral group of order $8$. The following lemma classifies the $w$-conjugacy classes of subgroups of $W_4$.

\begin{lemma}
  \label{lemma:w-conj-w4}
  There are exactly $3$ transitive and $7$ intransitive subgroups of $W_4$ and each is recorded in \Cref{tab:W4-and-W6-subgroups}. Moreover, the only distinct subgroups which are $w_A$-conjugate are:
  \begin{enumerate}
  \item
    $\wl{C2.4.nt.b.1}$ and $\wl{C2.4.nt.b.2}$ when $A$ is either ordinary or supersingular, and
  \item
    $\wl{C2.4.nt.c.1}$ and $\wl{C2.4.nt.c.2}$ when $A$ is supersingular.
  \end{enumerate}
  Each $w_A$-conjugacy class gives rise to an isomorphism class of weighted permutation representation, and the angle ranks of these $w_A$-conjugacy classes $\cclass{} \subset W_4$ are recorded in Tables~\ref{tab:dim2-A}--\ref{tab:dim2-C-ns}.
\end{lemma}

\begin{proof}
  The first claim then follows by a direct calculation. The angle ranks are computed using our implementation of \Cref{prop:newton-matrix-rank} in the file \href{https://github.com/sarangop1728/Galois-Frob-Polys/blob/main/src/weighted-perm-rep.m}{\texttt{src/weighted-perm-rep.m}} of our \GitHub repository \cite{OurElectronic}.
\end{proof}

To prove \Cref{thm:main-thm-surfaces} it suffices to show that:
\begin{enumerate}
    \item the cases we claim do not occur, actually do not occur; and
    \item in the cases that do occur, the weighted permutation representation determines whether the abelian surface is geometrically simple.
\end{enumerate}
In the remaining cases, we provide an example in Tables~\ref{tab:dim2-A}--\ref{tab:dim2-C-ns}, which realizes the given weighted permutation representation. We remark that the angle ranks displayed in the LMFDB are numerical approximations, but we verify these examples explicitly via a slower (but deterministic) algorithm; see the file \href{https://github.com/sarangop1728/Galois-Frob-Polys/blob/main/tables/verify-angle-rank.m}{\texttt{tables/verify-angle-rank.m}} in our \GitHub repository \cite{OurElectronic}.

\subsection{Proof of \texorpdfstring{\Cref{thm:main-thm-surfaces}}{Theorem ??} in the simple case}
For a simple abelian surface, the permutation representation acts transitively except when $A$ is supersingular with real Frobenius eigenvalues. 

\subsubsection{The ordinary case}
In this case, every possible transitive weighted permutation representation occurs. Therefore, to prove the theorem in this case, it suffices to show that a simple ordinary abelian surface is not geometrically simple if and only if its weighted permutation representation is \wl{V4.4.t.a.1}. 

\begin{proposition}
    \label{lemma:simple-splits-power-of-EC}
    A simple abelian variety has a unique simple factor over every finite extension of the base field. In particular, if $A$ is simple of prime dimension and it splits over a finite extension of $\FF_q$, then it does so as the power of an elliptic curve.
\end{proposition}
\begin{proof}
    This follows immediately from \cite[Proposition 1.2.6.1]{chai-conrad-oort}.
\end{proof}

\begin{coro}
    \label{coro:not-geom-simple=>angle-rank-1}
    If $A$ is a simple abelian variety of prime dimension which is not geometrically simple and not supersingular, then $A$ is ordinary and has angle rank $1$.
\end{coro}
\begin{proof}
    Since $A$ is simple but not geometrically simple by \Cref{lemma:simple-splits-power-of-EC} there exists an extension $\FF_{q^k}$ of $\FF_q$ over which $A$ becomes isogenous to $E^g$ for some ordinary elliptic curve $E/\FF_{q^k}$ (in particular $A$ is ordinary). Angle rank is invariant under base change, so it follows that $\delta_A = 1$.    
\end{proof}

\begin{lemma}
    \label{lemma:abs-simp-ord=>max-angle-rank}
    An ordinary geometrically simple abelian surface $A$ has angle rank $2$.
\end{lemma}
\begin{proof}
    We prove the contrapositive. Assume that $\delta_A < 2$. This implies that there exists a multiplicative relation among the normalized Frobenius eigenvalues $u_1^{r_1}u_2^{r_2} = 1$. Note that $r_1r_2 = 0$ implies that some $u_i$ is a root of unity, which would imply that $A$ is not ordinary. Thus, we have that both $r_1$ and $r_2$ are nonzero integers. Let $\alpha_1$ and $\alpha_2$ be the two Frobenius eigenvalues with $\nu(\alpha_1) = \nu(\alpha_2) = 0$ (i.e., $\alpha_i = u_i\sqrt{q}$). From the multiplicative relation we deduce that $r_1 = -r_2$, so that $(u_1/u_2)^{r_1} = 1$ and $\alpha_2 = \zeta_k\alpha_1$ where $\zeta_k$ is a $k^{\text{th}}$ root of unity. The Frobenius eigenvalues of the base change of $A$ to $\FF_{q^k}$ are precisely the $k$-th powers of the Frobenius eigenvalues of $A$. Because $\alpha_1^k = \alpha_2^k$, and $e = 1$ for ordinary abelian varieties, the Honda--Tate theorem implies that the base change of $A$ to $\FF_{q^k}$ is isogenous to the square of an elliptic curve, so $A$ is not geometrically simple.
\end{proof}

\begin{lemma} Let $A$ be a simple ordinary abelian surface. Then, exactly one of the following conditions holds.
\begin{enumerate}[label=(\arabic*)]
\item \label{enum:o-dim2-gs} $A$ is geometrically simple and $G_A \cong C_4$ or $W_4$.
\item \label{enum:o-dim2-ngs} $A$ is not geometrically simple and $G_A \cong V_4$.
\end{enumerate}
\end{lemma}
\begin{proof} 
By \Cref{lemma:abs-simp-ord=>max-angle-rank} if $A$ is geometrically simple then $A$ has angle rank 2, and by \Cref{coro:not-geom-simple=>angle-rank-1} if $A$ is not geometrically simple then it has angle rank $1$. The claim follows from the angle ranks computed in \Cref{lemma:w-conj-w4} (and recorded in \Cref{tab:dim2-A}).
\end{proof}

\subsubsection{The almost ordinary case}
Every simple almost ordinary abelian surface is geometrically simple by \Cref{coro:not-geom-simple=>angle-rank-1}, so the following lemma completes the proof.

\begin{lemma}
\label{lemma:ao-dim2}
A simple almost ordinary abelian surface has Galois group $W_4$.
\end{lemma}

\begin{proof}
Suppose for the sake of contradiction that $A$ is an almost ordinary abelian variety with Galois group $C_4$ or $V_4$. Let $\div_A$ be the corresponding divisor map as discussed in \Cref{prop:div-properties}. By point \ref{enum:div-property-4}, we have $\# \cP_L \mid 2$, where $\cP_L$ is the set of primes above $p$ in $L$. From \Cref{prop:newton-matrix-rank}, we have $\delta_A = \rank(\div_A) -1 \leq \#\cP_L - 1$, which contradicts \Cref{tab:dim2-B}.
\end{proof}

\begin{remark}
  One can also see \Cref{lemma:ao-dim2} as follows: by the theory of Newton polygons and the Honda--Tate theorem, $P_A(T)$ has an irreducible linear and quadratic factor over $\QQ_p$, so the quartic extension $K/\QQ$ is not Galois.
\end{remark}

\subsubsection{The supersingular case} 
It follows from the Honda--Tate theorem that every supersingular abelian variety has angle rank zero, because it is geometrically isogenous to the power of an elliptic curve.  

A glance at \Cref{tab:dim2-C} shows that to prove the first claim in \Cref{thm:main-thm-surfaces} in the simple supersingular case, it suffices to show that the permutation representation \wl{W4.4.t.a.1} does not occur, but in this case the angle rank is $2$. The second claim in \Cref{thm:main-thm-surfaces}, in the case of simple supersingular surfaces, also follows immediately from the Honda--Tate theorem.

\subsection{Proof of \texorpdfstring{\Cref{thm:main-thm-surfaces}}{Theorem ??} in the non-simple case}
If $A$ is isogenous to the square of an elliptic curve, then the claim follows immediately from \Cref{lemma:main-thm-ecs}. Therefore suppose that $A$ is isogenous over $\Fq$ to a product $E_1 \times E_2$ of non-isogenous elliptic curves. Note that when $A$ is supersingular, there is nothing to show so it suffices to consider the ordinary and almost ordinary cases. Let $\alpha_i$ and $\oalpha_i$ be the Frobenius eigenvalues of $E_i$ for each $i = 1, 2$.

\subsubsection{Non-simple ordinary abelian surfaces}
In this case, both $E_1$ and $E_2$ are ordinary elliptic curves and the Frobenius eigenvalues $\alpha_i,\oalpha_i$ are not real, and therefore $\cclass{A}$ contains complex conjugation. It is now easy to see that the Galois group is $\wl{V4.4.nt.a.1}$ if $\QQ(\alpha_1) \neq \QQ(\alpha_2)$ and $\wl{C2.4.nt.a.1}$ otherwise.

\begin{remark}
    In fact, using Lemma 5.3 of \cite{KrajicekScanlon00}, it is possible to show that in this case, the two elliptic curves are geometrically isogenous if and only if the Galois group is $\wl{C2.4.nt.a.1}$, but we will not need this.
\end{remark}

\subsubsection{Non-simple almost ordinary abelian surfaces}
In this case we may assume without loss of generality that $E_1$ is ordinary and $E_2$ is supersingular. First note that $\cclass{A}$ cannot be the $w_A$-conjugate to $\wl{C2.4.nt.b.1}$ or $\wl{C1.4.nt.a.1}$ since in this case the Frobenius eigenvalues of $E_1$ are fixed by $G_A$, which cannot occur. 

\begin{lemma}
    Let $E_1$ and $E_2$ be elliptic curves over $\FF_q$. Suppose that $E_1$ is ordinary and $E_2$ is supersingular. Then, the corresponding number fields satisfy $K_1 \cap K_2 = \QQ$.
\end{lemma}
\begin{proof}
	If $K_2 = \QQ$ there is nothing to show. Suppose that this is not the case, and assume in search of a contradiction that $K_1 = K_2 = K$. Then $\alpha_1 = u \alpha_2$ for $u \in K$ of absolute value $1$. Note that $u$ can't be a root of unity, since that would imply that $\alpha_1$ is a supersingular $q$-Weil number. Since $\alpha_2$ is a supersingular $q$-Weil number, some power of $u$ is an \emph{algebraic integer}. This implies that $u$ is also an algebraic integer which has length $1$ in all complex embeddings, and thus $u \in \OO_K^\times$. But $K$ is quadratic imaginary, implying that $u$ is a root of unity, a contradiction.
\end{proof}

\FloatBarrier
\section{Abelian threefolds}\label{sec:threefolds}
\begin{theorem}
  \label{thm:main-thm-3folds}
  Let $A$ be a simple abelian threefold.
  \begin{itemize}
      \item \thmemph{Permutation representations in $W_6$}. If $A$ has permutation representation contained in $W_6$, then the possible images of the weighted permutation representation associated to a simple abelian threefold are given in Tables~\ref{tab:threefoldA}--\ref{tab:threefoldE}. Each table corresponds to one Newton polygon. The weighted permutation representation determines whether $A$ is geometrically simple or not, and this information can be found in the tables. 
      \item \thmemph{Permutation representations in $W_2$}. If the permutation representation of $A$ is not contained in $W_6$, then it is contained in $W_2$. In this case $e_A = 1$, where $e_A$ is as in the statement of the Honda Tate--theorem. Such abelian varieties have Galois group $C_2$, angle rank $1$, and Newton polygon type $(D)$. They are geometrically simple. 
  \end{itemize}
\end{theorem}

The second point follows from \cite{Xing1994}; see \cite[Theorem 6.1.1, Lemma 6.1.2]{APBS2023} for more detail. In the remainder of this section we assume all abelian threefolds in question are simple and have permutation representation contained in $W_6$. We first classify the weighted permutation representations that occur, and then in \Cref{subsec:geom-simple} classify whether weighted permutation representations that occur are geometrically simple or not. The rest of this section is dedicated to proving \Cref{thm:main-thm-3folds}.

\subsection{Signed permutations on three elements}\label{sec:signed-permutations-3}
The following lemma follows by a direct calculation in \texttt{Magma}. See the file \href{https://github.com/sarangop1728/Galois-Frob-Polys/blob/main/src/W2d-subgroups.m}{\texttt{src/W2d-subgroups.m}} in our \GitHub repository \cite{OurElectronic}.
\begin{lemma}
  \label{prop:trans-subs}
  There are exactly $10$ transitive subgroups in $W_6$ which contain the complex conjugation element $\iota \in W_6$, and each is listed in \Cref{tab:W4-and-W6-subgroups}. These $10$ subgroups are contained in exactly $4$ $W_6$-conjugacy classes, namely those of:
  \begin{enumerate}
  \item
    $W_6$,
  \item
    $C_2 \wr C_3$ (transitive label \texttt{6T6}) generated by $(123)(\bar{1}\bar{2}\bar{3})$, $(1\bar{1})$, $(2\bar{2}),$ and $(3\bar{3})$,
  \item
    $D_6$ generated by $(123)(\bar{1}\bar{2}\bar{3})$ and $(12)(\bar{1}\bar{2})$,
  \item
    $C_6$ generated by $(123\bar{1}\bar{2}\bar{3})$.
  \end{enumerate}
  Moreover, for each possible Newton polygon of an abelian threefold $A$, the $w_A$-conjugacy classes of transitive subgroups of $W_6$ are recorded in Tables~\ref{tab:threefoldA}--\ref{tab:threefoldE}.
\end{lemma}

Similarly to the case of surfaces, to prove \Cref{thm:main-thm-3folds} it suffices to show that:
\begin{enumerate}
    \item the cases we claim do not occur, actually do not occur; and
    \item in the cases that do occur, the weighted permutation representation determines whether the abelian surface is geometrically simple (this is in \Cref{subsec:geom-simple}).
\end{enumerate}
In the remaining cases, we provide an example in Tables~\ref{tab:threefoldA}--\ref{tab:threefoldE} which realizes the given weighted permutation representation. We now prove $(1)$.

\subsection{Proof of \texorpdfstring{\Cref{thm:main-thm-3folds}}{Theorem ??} in the ordinary case}\label{sec:o-3}
In this case, the table shows that every possible weighted permutation representation occurs, so there is nothing to prove. 

\subsection{Proof of \texorpdfstring{\Cref{thm:main-thm-3folds}}{Theorem ??} in the almost ordinary case}\label{sec:ao-3}

We first classify weighted permutation representations that occur, and then give proofs of geometric simplicity. The following two lemmas complete the classification of weighted permutation representations in the case of simple almost ordinary abelian threefolds.

\begin{lemma}
\label{lem:ao-not-C6}
An almost ordinary abelian threefold cannot have Galois group $C_6$.
\end{lemma}
\begin{proof}
Suppose for the sake of contradiction that $A$ is an almost ordinary abelian threefold with Galois group $G_A \cong C_6$. Note that all weighted permutation representations of $C_6$ are conjugate in the almost ordinary case, so it suffices to show that the image of the weighted permutation representation is not conjugate to $\wl{C6.6.t.a.2}$, which is generated by $\sigma = (123\bar{1}\bar{2}\bar{3})$. We now show that this weighted permutation representation doesn't lift to a divisor map with the properties listed in \Cref{prop:div-properties}.

By \Cref{prop:div-properties}\ref{enum:div-property-4}, we have $\# \cP_L \mid 3$ where $\cP_L$ is the set of primes of $K = L$ dividing $p$. But then the action of $G_A$ on $\cP_L$ factors through $C_3$ and in particular the sequence 
\[
(\nu(\alpha_1), \nu(\sigma \alpha_1), ..., \nu(\sigma^5 \alpha_1)) = \left(0, 0, \tfrac{1}{2}, 1, 1, \tfrac{1}{2} \right)
\]
should be $3$-periodic, a contradiction.
\end{proof}

\begin{lemma}
\label{lem:aoC6}
An almost ordinary abelian threefold with Galois group $D_6$ has angle rank $2$.
\end{lemma}

\begin{proof}
From \Cref{tab:threefoldB}, it suffices to show that the weighted permutation representation \wl{D6.6.t.a.1}, which has angle rank $3$, does not occur. Suppose for the sake of contradiction that $A$ is an almost ordinary abelian threefold with weighted permutation representation \wl{D6.6.t.a.1}. We exhibit a nontrivial multiplicative relation between the Frobenius eigenvalues (a contradiction to the angle rank being maximal) -- in particular we show that $\alpha_1\oalpha_3/q$ is a root of unity.

\newcounter{claimcount}\setcounter{claimcount}{1}
\claim[\theclaimcount]{$\#\cP_L = 6$.} \addtocounter{claimcount}{1}
    By \Cref{prop:div-properties}\ref{enum:div-property-4}, we have $\#\cP_L \mid 6$. The order $6$ cyclic subgroup of \wl{D6.6.t.a.1} is generated by $\sigma = (1 \bar{2} 3 \bar{1} 2 \bar{3})$. As in the proof of \Cref{lem:ao-not-C6}, because the sequence 
    \[
(\nu(\alpha_1), \nu(\sigma \alpha_1), ..., \nu(\sigma^5 \alpha_1 )) = \left( 0, 1, \tfrac{1}{2}, 1, 0, \tfrac{1}{2} \right)
\]
is not periodic, we must have $\#\cP_L = 6$.
    
\claim[\theclaimcount]{$\alpha_1 \oalpha_3/q$ is a root of unity.}
    Let $\mfp_1$ be the prime of $\OO_L$ corresponding to the valuation $\nu$ and let
    \[
        (\mfp_1,\op_2,\mfp_3,\op_1,\mfp_2,\op_3) = (\mfp_1, \sigma(\mfp_1),\dots,\sigma^5(\mfp_1)).
    \]
    Because $D_6$ has a unique transitive permutation representation on a $6$ element set, the action of $D_6$ on $\cP_L$ is exactly the rigid symmetries of the hexagon as depicted in \Cref{fig:hexagon}. Because
    \[
    (\nu(\alpha_1), \nu(\sigma \alpha_1), ..., \nu(\sigma^5 \alpha_1)) = \left( 0, 1, \tfrac{1}{2}, 1, 0, \tfrac{1}{2} \right),
    \]
    we have
    \[
        \div_A(\alpha_1) = en \big( [\op_3] + \tfrac{1}{2}[\mfp_2] + [\op_1] + \tfrac{1}{2} [\op_2] \big)
    \]
    where $q = p^n$ and $e$ is the ramification index of $p$ in $L$. Now consider the element $\tau = (1\bar{3})(2\bar{2})(\bar{1}3) \in \wl{D6.6.t.a.1}$ depicted in \Cref{fig:hexagon}. Note that $\tau$ is \emph{not} complex conjugation and in fact $\tau(\alpha_1) = \oalpha_3$. The action of $D_6$ on $\mathcal{P}_L$ allows us to compute $\div_A(\oalpha_3)$, and it is:
    \[
        \div_A(\oalpha_3) = \tau(\div_A(\alpha_1)) = en \left( [\mfp_1] + \tfrac{1}{2}[\op_2] + [\mfp_3] + \tfrac{1}{2}[\mfp_2] \right).
    \]
    We have $(\oalpha_3\alpha_1)\OO_L = q\OO_L$, so $\oalpha_3\alpha_1 = qu$ for some unit $u \in \OO_L^\times$. Because $\alpha_1$ and $\oalpha_3$ are both $q$-Weil numbers, the unit $u=\alpha_1\oalpha_3/q$ has absolute value $1$ over all complex places, thus it is a root of unity.
\end{proof}

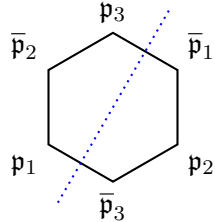
\begin{figure}[ht]
    \centering
    \scalebox{0.99}{
    \begin{tikzpicture}
        \draw[thick] 
        (30:1)  node[anchor=south west] {$\op_1$} -- 
        (90:1)  node[anchor=south] {$\mfp_3$} -- 
        (150:1) node[anchor=south east] {$\op_2$} -- 
        (210:1) node[anchor=north east] {$\mfp_1$} -- 
        (270:1) node[anchor=north] {$\op_3$} -- 
        (330:1) node[anchor=north west] {$\mfp_2$} -- 
        cycle;
        \draw[dotted, thick, blue] (60:1.5) -- (240:1.5);  
    \end{tikzpicture}
    }
    \caption{The action of $G_A$ on $\cP_L$. The permutation $\tau \in G_A$ is reflection with respect to the dotted line.}
    \label{fig:hexagon}
\end{figure}
\FloatBarrier

\subsection{Proof of \texorpdfstring{\Cref{thm:main-thm-3folds}}{Theorem ??} in the Newton polygon (C) case}

The following lemma completes the proof. 
\begin{lemma}
  An abelian threefold with Newton polygon (C) cannot have Galois group $G_A \cong C_6$.
\end{lemma}
\begin{proof}
  Suppose for the sake of contradiction that $G_A \cong C_6$. It suffices to show that the image of the weighted permutation representation is not conjugate to $\wl{C6.6.t.a.2}$, which is generated by $\sigma = (123\bar{1}\bar{2}\bar{3})$. By \Cref{prop:div-properties}\ref{enum:div-property-4}, we have $\# \cP_L \mid 3$. But then the action of $G_A$ on $\cP_L$ factors through $C_3$ and in particular the sequence 
  \[
    (\nu(\alpha_1), \nu(\sigma \alpha_1), ..., \nu(\sigma^5 \alpha_1)) = \left(0, \tfrac{1}{2}, \tfrac{1}{2}, 1, \tfrac{1}{2}, \tfrac{1}{2} \right)
  \]
  should be $3$-periodic, a contradiction.
\end{proof}

\subsection{Proof of \texorpdfstring{\Cref{thm:main-thm-3folds}}{Theorem ??} in the Newton polygon (D) case}\label{sec:non-ss-prank0-3}
The following lemmas complete the proof. 
\begin{lemma}
A type (D) abelian threefold cannot have weighted permutation representation whose image is $w_A$-conjugate to either $\wl{C6.6.t.a.2}$ or $\wl{D6.6.t.a.4}$.
\end{lemma}

\begin{proof}
  Suppose for the sake of contradiction that the image of the weighted permutation representation is $w_A$-conjugate to either \wl{C6.6.t.a.2} or \wl{D6.6.t.a.4}. By \Cref{prop:div-properties}\ref{enum:div-property-4}, we have $\# \cP_L \mid 4$ where $\cP_L$ is the set of primes of $L$ dividing $p$. But then the action of the order $6$ element $\sigma = (123\bar{1}\bar{2}\bar{3}) \in G_A$ on $\cP_L$ factors through $C_2$. But the sequence 
  \[
    (\nu(\alpha_1), \nu(\sigma\alpha_1), ..., \nu(\sigma^5\alpha_1)) =  \left( \tfrac{1}{3}, \tfrac{1}{3}, \tfrac{1}{3}, \tfrac{2}{3},
      \tfrac{2}{3}, \tfrac{2}{3} \right)
  \]
  is not $2$-periodic, a contradiction.
\end{proof}

\subsection{Proof of \texorpdfstring{\Cref{thm:main-thm-3folds}}{Theorem ??} in the supersingular case}\label{sec:ss-3}
\begin{lemma}
  The sextic field generated by the Frobenius eigenvalues of a simple supersingular abelian threefold must be $\QQ(\zeta_7)$ or $\QQ(\zeta_9)$, both of which have Galois group $C_6$.
\end{lemma}

\begin{proof}
  We follow the argument in \cite[Proposition~2.1]{NartRitzenthaler08} (note that the characteristic of the base field in \textit{loc. cit.} is $2$). Supersingular abelian varieties have angle rank $0$, so the sextic field $K$ (as defined in \Cref{def:K}) is of the form $K \cong \QQ(\zeta_m \sqrt{q})$ where $\zeta_m$ is a primitive $m^\text{th}$-root of unity. Choose $m$ to be the smallest integer such that $\zeta_m\sqrt{q}$ generates $K$. 

  If $q$ is a square, then the only cyclotomic fields of degree $6$ are $\QQ(\zeta_7)$ or $\QQ(\zeta_9)$, so we are done. 

  Suppose now that $q$ is not a square. If $m$ is odd then $K$ contains the field $\QQ(\zeta_m)$. Thus $[\QQ(\zeta_m) : \QQ] \leq 6$ and $m = 3,7,9$. If $m = 3$ then $[K : \QQ] \leq [\QQ(\zeta_3,\sqrt{q}) : \QQ] = 4$. If $m = 7,9$, then $6 = [K : \QQ] \geq  [\QQ(\zeta_m) : \QQ] = 6$, so $K = \QQ(\zeta_m)$.

  Now suppose $m = 2n$ is even. Since $K$ contains $\QQ(\zeta_m^2) = \QQ(\zeta_{n})$ we have $n = 3,4,6,7,9,14$. If $n = 7,9,14$ then $6 = [K : \QQ] \geq [\QQ(\zeta_{n}) : \QQ] = 6$, so $K = \QQ(\zeta_{n})$. If $n = 3$, then $[K : \QQ] \leq [\QQ(\zeta_6, \sqrt{q}) : \QQ] = 4$, contradiction. If $n = 6$, then $K = \QQ(\zeta_{12} \sqrt{q})$. However, this is a quadratic extension of the quadratic field $\QQ(\zeta_6)$, so $[K : \QQ] = 4$, which is a contradiction.
\end{proof}

\subsection{When are simple abelian threefolds geometrically simple?}
\label{subsec:geom-simple}
By the Honda--Tate theorem every supersingular abelian variety is not geometrically simple. Moreover, by \Cref{coro:not-geom-simple=>angle-rank-1} a simple ordinary threefold which is neither supersingular nor geometrically simple must be ordinary. 

Suppose that $A$ is a simple ordinary abelian threefold. In this case all possible weighted permutation representations occur, and it suffices to show that if $A$ is geometrically simple if it has angle rank $3$ and not geometrically simple if it has angle rank $1$.

\begin{lemma}
\label{lem:geom-simple<=>max-angle-rank}
Let $A$ be a simple ordinary abelian threefold. Then the angle rank of $A$ is $3$ if $A$ is geometrically simple and $1$ otherwise.
\end{lemma}
\begin{proof}
  If $A$ is geometrically simple, then $\delta_A = 3$ by \cite[Lemma~6.2.2]{APBS2023}. If $A$ is not geometrically simple then $A$ has angle rank $1$ by \Cref{coro:not-geom-simple=>angle-rank-1}.
\end{proof}

\FloatBarrier
\section{Inverse Galois Problems}\label{sec:questions}
We state a slight refinement of a conjecture Dupuy, Kedlaya, Roe, and Vincent \cite[Conjecture 2.7]{DupuyKedlayaRoeVincent21}.

\begin{conjecture}[Refined inverse Galois problem for abelian varieties]
  \label{conj:IGP}
  Fix a prime number $p$ and let $G \subset W_{2d}$ be a transitive subgroup containing the complex conjugation element. Then:
  \begin{enumerate}
  \item 
    there exists an integer $r \geq 1$ and a simple abelian variety $A/\FF_{p^r}$ of dimension $d$ such that $G$ is $w_A$-conjugate to the image of the weighted permutation representation associated to $A$, and
  \item 
    the abelian variety $A$ may be taken to be ordinary.
  \end{enumerate}
  In particular, $G$ is isomorphic to the Galois group of some abelian variety $A$.
\end{conjecture}

We note that \Cref{conj:IGP} is a very strong statement. In particular it implies the inverse Galois problem holds even when we are restricted to totally real fields.

\begin{proposition}
  Assume \Cref{conj:IGP}, then the inverse Galois problem holds for totally real fields. More precisely, let $d \geq 1$ and let $G \subset S_d$ be a transitive subgroup, then $G$ is the Galois group of a polynomial $P^+(T)$ of degree $d$ over $\QQ$ whose splitting field is totally real.
\end{proposition}

\begin{proof}
  Let $G \subset S_d$ be a transitive subgroup, and consider the group $\widetilde{G} = C_2 \wr G$ equipped with its natural embedding $\widetilde{G} \hookrightarrow W_{2d}$. Note that $\widetilde{G}$ is a transitive subgroup of $W_{2d}$ which contains the complex conjugation element, so by assumption it occurs as the Galois group of a $q$-Weil polynomial $P_A(T)$. Let $P_A^+(T)$ be the \cdef{trace polynomial} of $P_A(T)$ defined by the equation
  \begin{equation*}
    P_A^+(T) = \prod_{\alpha}(T - (\alpha + \oalpha)),
  \end{equation*}
  where $\alpha$ ranges over the roots of $P_A(T)$. It follows by construction that $\Gal(P_A^+(T))$ is isomorphic to $G$ and that every root of $P_A^+(T)$ is real.
\end{proof}

\appendix\section{Tables}
In \Cref{tab:W4-and-W6-subgroups} we record our labelling convention for subgroups of $W_4$ and $W_6$. In particular, we list every subgroup of $W_4$ and every transitive subgroup of $W_6$ containing complex conjugation.

\subsection{Abelian surfaces}
Tables~\ref{tab:dim2-A}--\ref{tab:dim2-C-ns} record the possible subgroups $G \subset W_4$ which may occur as the ($w_A$-conjugacy class of the) images of weighted permutation representations associated to abelian surfaces $A$. For each case which does occur, we provide an example from the LMFDB~\cite{lmfdb}. The tables are separated by Newton polygon (according to the conventions in \Cref{fig:flowchart2-simple}) and by whether $A$ is simple.

\subsubsection{Simple abelian surfaces}
These cases are treated in Tables~\ref{tab:dim2-A}--\ref{tab:dim2-C}. We do not list intransitive subgroups which do not occur as the image of the weighted permutation representation associated to a simple abelian surface -- note that an intransitive subgroup can only occur for the supersingular abelian surface in \Cref{lem:real-eigenvalues}\ref{enum:complex-ECs}. In each case we record whether every isogeny class of abelian surfaces with the recorded image of the weighted permutation representation is geometrically simple.

\subsubsection{Non-simple abelian surfaces}
These cases are treated in Tables~\ref{tab:dim2-A-ns}--\ref{tab:dim2-C-ns}. Since transitive subgroups of $W_4$ cannot occur as the image of the weighted permutation representation associated to an abelian surface, we do not record them.

\subsection{Abelian threefolds}
Tables~\ref{tab:threefoldA}--\ref{tab:threefoldE} record the possible subgroups $G \subset W_6$ which may occur as the ($w_A$-conjugacy class of the) images of weighted permutation representations associated to simple abelian threefolds $A$. For each case which does occur, we provide an example from the LMFDB~\cite{lmfdb}. The tables are separated by Newton polygon (according to the conventions in \Cref{fig:flowchart3}).
\FloatBarrier

{
  \vfill
  \begin{table}[ht]
    \setlength{\arrayrulewidth}{0.3mm} 
    \setlength{\tabcolsep}{5pt}
    \renewcommand{\arraystretch}{1.15}
    \centering
    \begin{tabular}{|c|C{4.5cm}||c|C{4.5cm}|}
      \hline
      \rowcolor{headercolor}
      Label of $G$       & Generators of $G$             & Label of $G$       & Generators of $G$                                                        \\
      \hline
      \mkwl{W4.4.t.a.1}  & $W_4$                         & \mkwl{W6.6.t.a.1}  & $W_6$                                                                    \\ \hline
      \mkwl{V4.4.t.a.1}  & $\iota, (12)(\bar{1}\bar{2})$ & \mkwl{6T6.6.t.a.1} & $(123)(\bar{1}\bar{2}\bar{3})$, $(1\bar{1})$, $(2\bar{2})$, $(3\bar{3})$ \\ \hline
      \mkwl{C4.4.t.a.1}  & $(12\bar{1}\bar{2})$          & \mkwl{D6.6.t.a.1}  & $(1\bar{2}3\bar{1}2\bar{3}), (23)(\bar{2}\bar{3})$                       \\ \hline
      \mkwl{V4.4.nt.a.1} & $(1\bar{1}), (2\bar{2})$      & \mkwl{D6.6.t.a.2}  & $(12\bar{3}\bar{1}\bar{2}3)$, $(23)(\bar{2}\bar{3})$                     \\ \hline
      \mkwl{C2.4.nt.a.1} & $\iota$                       & \mkwl{D6.6.t.a.3}  & $(1\bar{2}\bar{3} \bar{1}23)$, $(2\bar{3})(\bar{2}3)$                    \\ \hline
      \mkwl{C2.4.nt.b.1} & $(1\bar{1})$                  & \mkwl{D6.6.t.a.4}  & $(123\bar{1}\bar{2}\bar{3})$, $(2\bar{3})(\bar{2}3)$                     \\ \hline
      \mkwl{C2.4.nt.b.2} & $(2\bar{2})$                  & \mkwl{C6.6.t.a.1}  & $(1\bar{2}3\bar{1}2\bar{3})$                                             \\ \hline
      \mkwl{C2.4.nt.c.1} & $(12)(\bar{1}\bar{2})$        & \mkwl{C6.6.t.a.2}  & $(123\bar{1}\bar{2}\bar{3})$                                             \\ \hline
      \mkwl{C2.4.nt.c.2} & $(1\bar{2})(\bar{1}2)$        & \mkwl{C6.6.t.a.3}  & $(12\bar{3}\bar{1}\bar{2}3)$                                             \\ \hline
      \mkwl{C1.4.nt.a.1} & $\operatorname{id}$           &\mkwl{C6.6.t.a.4}   & $(1\bar{2}\bar{3}\bar{1}23)$                                             \\ \hline
    \end{tabular}
    \caption{Labels for subgroups of $W_4$ (left) and subgroups of $W_6$ (right).}
    \label{tab:W4-and-W6-subgroups}
  \end{table}

  \begin{table}[ht]
    \setlength{\arrayrulewidth}{0.3mm} 
    \setlength{\tabcolsep}{5pt}
    \renewcommand{\arraystretch}{1.2}
    \centering
    \begin{tabular}{|c|c|c|c|c|}
      \hline
      \rowcolor{headercolor}
      $w_A$-conjugacy class & Angle rank & Occurs & Geometrically simple & Example(s)        \\
      \hline
      \wl{W4.4.t.a.1}       & $2$        & Yes    & Yes                 & \avlink{2.2.ac_d} \\ \hline
      \wl{V4.4.t.a.1}       & $1$        & Yes    & No                  & \avlink{2.2.ad_f} \\ \hline
      \wl{C4.4.t.a.1}       & $2$        & Yes    & Yes                 & \avlink{2.3.ad_f} \\ \hline
    \end{tabular}
    \caption{The images of the weighted permutation representations associated to a simple ordinary abelian surface (Newton polygon (A) in \Cref{fig:flowchart2-simple}).}
    \label{tab:dim2-A}
  \end{table}
  \vfill
}

\begin{table}[ht]
  \setlength{\arrayrulewidth}{0.3mm} 
  \setlength{\tabcolsep}{5pt}
  \renewcommand{\arraystretch}{1.2}
  \centering
  \begin{tabular}{|c|c|c|c|c|}
    \hline
    \rowcolor{headercolor}
    $w_A$-conjugacy class & Angle rank & Occurs & Geometrically simple & Example(s)        \\
    \hline
    \wl{W4.4.t.a.1}       & $2$        & Yes    & Yes & \avlink{2.2.ab_a} \\ \hline
    \wl{V4.4.t.a.1}       & $2$        & No     &     &                   \\ \hline
    \wl{C4.4.t.a.1}       & $2$        & No     &     &                   \\ \hline
  \end{tabular}
    \caption{The images of the weighted permutation representations associated to simple almost ordinary abelian surfaces (Newton polygon (B) in \Cref{fig:flowchart2-simple}).}
  \label{tab:dim2-B}
\end{table}

\begin{table}[p]
  \setlength{\arrayrulewidth}{0.3mm} 
  \setlength{\tabcolsep}{5pt}
  \renewcommand{\arraystretch}{1.2}
  \centering
  \begin{tabular}{|c|c|c|c|c|}
    \hline
    \rowcolor{headercolor}
    $w_A$-conjugacy class & Angle rank           & Occurs               & Geometrically simple & Example(s)                         \\
    \hline
    \wl{W4.4.t.a.1}       & $2$                  & No                   &                      &                                    \\ \hline
    \wl{V4.4.t.a.1}       & $0$                  & Yes                  & No                   & \avlink{2.2.ac_c}                  \\ \hline
    \wl{C4.4.t.a.1}       & $0$                  & Yes                  & No                   & \avlink{2.4.ac_e}                  \\ \hline
    \wl{C2.4.nt.c.1}      & \multirow{2}{*}{$0$} & \multirow{2}{*}{Yes} & \multirow{2}{*}{No}  & \multirow{2}{*}{\avlink{2.2.a_ae}} \\
    \wl{C2.4.nt.c.2}      &                      &                      &                      &                                    \\ \hline
  \end{tabular}
    \caption{The images of the weighted permutation representations associated to simple supersingular abelian surfaces (Newton polygon (C) in \Cref{fig:flowchart2-simple}).}
  \label{tab:dim2-C}
\end{table}

\begin{table}[p]
  \setlength{\arrayrulewidth}{0.3mm} 
  \setlength{\tabcolsep}{5pt}
  \renewcommand{\arraystretch}{1.2}
  \centering
  \begin{tabular}{|c|c|c|c|}
    \hline
    \rowcolor{headercolor}
    $w_A$-conjugacy class & Angle rank           & Occurs              &  Example(s)       \\ \hline
    \wl{V4.4.nt.a.1}      & $1$                  & Yes                 & \avlink{2.3.ad_i} \\ \hline
    \wl{C2.4.nt.a.1}      & $1$                  & Yes                 & \avlink{2.2.a_d}  \\ \hline
    \wl{C2.4.nt.b.1}      & \multirow{2}{*}{$1$} & \multirow{2}{*}{No} &                   \\
    \wl{C2.4.nt.b.2}      &                      &                     &                   \\ \hline
    \wl{C2.4.nt.c.1}      & $1$                  &  No                 &                   \\ \hline
    \wl{C2.4.nt.c.2}      & $1$                  &  No                 &                   \\ \hline
    \wl{C1.4.nt.a.1}      & $0$                  &  No                 &                   \\ \hline
  \end{tabular}
    \caption{The images of the weighted permutation representations associated to non-simple ordinary abelian surfaces (Newton polygon (A) in \Cref{fig:flowchart2-non-simple}).}
  \label{tab:dim2-A-ns}
\end{table}

\begin{table}[p]
  \setlength{\arrayrulewidth}{0.3mm} 
  \setlength{\tabcolsep}{5pt}
  \renewcommand{\arraystretch}{1.2}
  \centering
  \begin{tabular}{|c|c|c|c|}
    \hline
    \rowcolor{headercolor}
    $w_A$-conjugacy class & Angle rank & Occurs & Example(s)        \\
    \hline
    \wl{V4.4.nt.a.1}      & $1$        & Yes    & \avlink{2.2.ad_g} \\ \hline
    \wl{C2.4.nt.a.1}      & $1$        & No     &                   \\ \hline
    \wl{C2.4.nt.b.1}      & $1$        & No     &                   \\ \hline
    \wl{C2.4.nt.b.2}      & $1$        & Yes    & \avlink{2.4.ah_u} \\ \hline
    \wl{C2.4.nt.c.1}      & $1$        &  No    &                   \\ \hline
    \wl{C2.4.nt.c.2}      & $1$        &  No    &                   \\ \hline
    \wl{C1.4.nt.a.1}      & $0$        &  No    &                   \\ \hline
  \end{tabular}
    \caption{The images of the weighted permutation representations associated to a non-simple almost ordinary abelian surfaces (Newton polygon (B) in \Cref{fig:flowchart2-non-simple}).}
  \label{tab:dim2-B-ns}
\end{table}

\begin{table}[p]
  \setlength{\arrayrulewidth}{0.3mm} 
  \setlength{\tabcolsep}{5pt}
  \renewcommand{\arraystretch}{1.2}
  \centering
  \begin{tabular}{|c|c|c|c|c|}
    \hline
    \rowcolor{headercolor}
    $w_A$-conjugacy class & Angle rank           & Occurs               & Example(s)                          \\ \hline
    \wl{V4.4.nt.a.1}      & $0$                  & Yes                  &  \avlink{2.2.ac_e}                  \\ \hline
    \wl{C2.4.nt.a.1}      & $0$                  & Yes                  &  \avlink{2.2.a_a}                   \\ \hline
    \wl{C2.4.nt.b.1}      & \multirow{2}{*}{$0$} & \multirow{2}{*}{Yes} &  \multirow{2}{*}{\avlink{2.4.ag_q}} \\
    \wl{C2.4.nt.b.2}      &                      &                      &                                     \\ \hline
    \wl{C2.4.nt.c.1}      & \multirow{2}{*}{$0$} & \multirow{2}{*}{No}  &                                     \\
    \wl{C2.4.nt.c.2}      &                      &                      &                                     \\ \hline
    \wl{C1.4.nt.a.1}      & $0$                  & Yes                  & \avlink{2.4.a_ai}                   \\ \hline
  \end{tabular}
    \caption{The images of the weighted permutation representations associated to non-simple supersingular abelian surfaces  (Newton polygon (C) in \Cref{fig:flowchart2-non-simple}).}
  \label{tab:dim2-C-ns}
\end{table}

\begin{table}[p]
  \setlength{\arrayrulewidth}{0.3mm} 
  \setlength{\tabcolsep}{5pt}
  \renewcommand{\arraystretch}{1.2}
  \centering
  \begin{tabular}{|c|c|c|c|c|}
    \hline
    \rowcolor{headercolor}
    $w_A$-conjugacy class & Angle rank           & Occurs               & Geometrically simple & Example(s)                              \\ \hline
    \wl{W6.6.t.a.1}       & $3$                  & Yes                  & Yes                  & \avlink{3.2.ad_f_ah}                    \\ \hline
    \wl{6T6.6.t.a.1}      & $3$                  & Yes                  & Yes                  & \avlink{3.2.ad_g_aj}                    \\ \hline
    \wl{D6.6.t.a.1}       & $1$                  & Yes                  & No                   & \avlink{3.2.a_a_ad}                     \\ \hline
    \wl{D6.6.t.a.2}       & \multirow{3}{*}{$3$} & \multirow{3}{*}{Yes} & \multirow{3}{*}{Yes} & \multirow{3}{*}{\avlink{3.2.ac_a_d}}    \\
    \wl{D6.6.t.a.3}       &                      &                      &                      &                                         \\
    \wl{D6.6.t.a.4}       &                      &                      &                      &                                         \\ \hline
    \wl{C6.6.t.a.1}       & $1$                  & Yes                  & No                   & \avlink{3.2.ae_j_ap}                    \\ \hline
    \wl{C6.6.t.a.2}       & \multirow{3}{*}{$3$} & \multirow{3}{*}{Yes} & \multirow{3}{*}{Yes} & \multirow{3}{*}{\avlink{3.7.ak_bw_afv}} \\
    \wl{C6.6.t.a.3}       &                      &                      &                      &                                         \\
    \wl{C6.6.t.a.4}       &                      &                      &                      &                                         \\ \hline
  \end{tabular}
  \caption{The images of the weighted permutation representations associated to simple ordinary abelian threefolds (Newton polygon (A) in \Cref{fig:flowchart3}).}
  \label{tab:threefoldA}
\end{table}

\begin{table}[p]
  \setlength{\arrayrulewidth}{0.3mm} 
  \setlength{\tabcolsep}{5pt}
  \renewcommand{\arraystretch}{1.2}
  \centering
  \begin{tabular}{|c|c|c|c|c|}
    \hline
    \rowcolor{headercolor}
    $w_A$-conjugacy class & Angle rank           & Occurs               & Geometrically simple & Example                              \\\hline
    \wl{W6.6.t.a.1}       & $3$                  & Yes                  & Yes                  & \avlink{3.2.ab_ab_c}                 \\ \hline
    \wl{6T6.6.t.a.1}      & $3$                  & Yes                  & Yes                  & \avlink{3.4.ac_ab_g}                 \\ \hline
    \wl{D6.6.t.a.1}       & \multirow{2}{*}{$3$} & \multirow{2}{*}{No}  &                      &                                      \\
    \wl{D6.6.t.a.3}       &                      &                      &                      &                                      \\ \hline
    \wl{D6.6.t.a.2}       & \multirow{2}{*}{$2$} & \multirow{2}{*}{Yes} & \multirow{2}{*}{Yes} & \multirow{2}{*}{\avlink{3.2.ac_b_a}} \\
    \wl{D6.6.t.a.4}       &                      &                      &                      &                                      \\ \hline
    \wl{C6.6.t.a.1}       & \multirow{2}{*}{$3$} & \multirow{2}{*}{No}  &                      &                                      \\
    \wl{C6.6.t.a.4}       &                      &                      &                      &                                      \\ \hline
    \wl{C6.6.t.a.2}       & \multirow{2}{*}{$2$} & \multirow{2}{*}{No}  &                      &                                      \\
    \wl{C6.6.t.a.3}       &                      &                      &                      &                                      \\ \hline
  \end{tabular}
  \caption{The images of the weighted permutation representations associated to simple almost ordinary abelian threefold (Newton polygon (B) in \Cref{fig:flowchart3}).}
  \label{tab:threefoldB}
\end{table}

\begin{table}[p]
  \setlength{\arrayrulewidth}{0.3mm} 
  \setlength{\tabcolsep}{5pt}
  \renewcommand{\arraystretch}{1.2}
  \centering
  \begin{tabular}{|c|c|c|c|c|}
    \hline
    \rowcolor{headercolor}
    $w$-conjugacy class & Angle rank           & Occurs               & Geometrically simple  & Example                               \\
    \hline
    \wl{W6.6.t.a.1}     & $3$                  & Yes                  & Yes                  & \avlink{3.2.ab_a_a}                   \\ \hline
    \wl{6T6.6.t.a.1} & $3$                  & Yes                  & Yes                  & \avlink{3.4.ab_c_a}                   \\ \hline
    \wl{D6.6.t.a.1}     & \multirow{4}{*}{$3$} & \multirow{4}{*}{Yes} & \multirow{4}{*}{Yes} & \multirow{4}{*}{\avlink{3.4.ab_a_ae}} \\
    \wl{D6.6.t.a.2}     &                      &                      &                       &                                       \\
    \wl{D6.6.t.a.3}     &                      &                      &                       &                                       \\
    \wl{D6.6.t.a.4}     &                      &                      &                       &                                       \\ \hline
    \wl{C6.6.t.a.1}     & \multirow{4}{*}{$3$} & \multirow{4}{*}{No}  &                       &                                       \\
    \wl{C6.6.t.a.2}     &                      &                      &                       &                                       \\
    \wl{C6.6.t.a.3}     &                      &                      &                       &                                       \\
    \wl{C6.6.t.a.4}     &                      &                      &                       &                                       \\ \hline
  \end{tabular}
  \caption{The images of the weighted permutation representations associated to a simple abelian threefolds with Newton polygon (C) in \Cref{fig:flowchart3}.}
  \label{tab:threefoldC}
\end{table}

\begin{table}[p]
  \setlength{\arrayrulewidth}{0.3mm} 
  \setlength{\tabcolsep}{5pt}
  \renewcommand{\arraystretch}{1.2}
  \centering
  \begin{tabular}{|c|c|c|c|c|}
    \hline
    \rowcolor{headercolor}
    $w$-conjugacy class & Angle rank           & Occurs              & Geometrically simple & Example              \\
    \hline
    \wl{W6.6.t.a.1}     & $3$                  & Yes                 & Yes                 & \avlink{3.2.ac_c_ac} \\ \hline
    \wl{6T6.6.t.a.1}    & $3$                  & Yes                 & Yes                 & \avlink{3.3.ad_j_ap} \\ \hline
    \wl{D6.6.t.a.1}     & $1$                  & Yes                 & Yes                 & \avlink{3.2.a_a_ac}  \\ \hline
    \wl{D6.6.t.a.2}     & \multirow{3}{*}{$3$} & \multirow{3}{*}{No} &                      &                      \\
    \wl{D6.6.t.a.3}     &                      &                     &                      &                      \\
    \wl{D6.6.t.a.4}     &                      &                     &                      &                      \\ \hline
    \wl{C6.6.t.a.1}     & $1$                  & Yes                 & Yes                 & \avlink{3.7.a_a_abj} \\ \hline
    \wl{C6.6.t.a.2}     & \multirow{3}{*}{$3$} & \multirow{3}{*}{No} &                      &                      \\
    \wl{C6.6.t.a.3}     &                      &                     &                      &                      \\
    \wl{C6.6.t.a.4}     &                      &                     &                      &                      \\ \hline
  \end{tabular}
  \caption{The images of the weighted permutation representations associated to simple abelian threefolds with Newton polygon (D) in \Cref{fig:flowchart3}.}
  \label{tab:threefoldD}
\end{table}
\FloatBarrier

\begin{table}[H]
  \setlength{\arrayrulewidth}{0.3mm} 
  \setlength{\tabcolsep}{5pt}
  \renewcommand{\arraystretch}{1.2}
  \centering
  \begin{tabular}{|c|c|c|c|c|}
    \hline
    \rowcolor{headercolor}
    $w$-conjugacy class & Angle rank           & Occurs               & Geometrically simple & Example                              \\
    \hline
    \wl{W6.6.t.a.1}     & $0$                  & No                   &                      &                                      \\ \hline
    \wl{6T6.6.t.a.1}    & $0$                  & No                   &                      &                                      \\ \hline
    \wl{D6.6.t.a.1}     & \multirow{4}{*}{$0$} & \multirow{4}{*}{No}  &                      &                                      \\
    \wl{D6.6.t.a.2}     &                      &                      &                      &                                      \\
    \wl{D6.6.t.a.3}     &                      &                      &                      &                                      \\
    \wl{D6.6.t.a.4}     &                      &                      &                      &                                      \\ \hline
    \wl{C6.6.t.a.1}     & \multirow{4}{*}{$0$} & \multirow{4}{*}{Yes} & \multirow{4}{*}{No} & \multirow{4}{*}{\avlink{3.3.a_a_aj}} \\
    \wl{C6.6.t.a.2}     &                      &                      &                      &                                      \\
    \wl{C6.6.t.a.3}     &                      &                      &                      &                                      \\
    \wl{C6.6.t.a.4}     &                      &                      &                      &                                      \\ \hline
  \end{tabular}
  \caption{The images of the weighted permutation representations associated to simple almost ordinary abelian threefolds (Newton polygon (E) in \Cref{fig:flowchart3}).}
  \label{tab:threefoldE}
\end{table}

\FloatBarrier
\providecommand{\bysame}{\leavevmode\hbox to3em{\hrulefill}\thinspace}
\providecommand{\MR}{\relax\ifhmode\unskip\space\fi MR }
\providecommand{\MRhref}[2]{%
  \href{http://www.ams.org/mathscinet-getitem?mr=#1}{#2}
}
\providecommand{\href}[2]{#2}

\end{document}